\documentclass[10pt]{amsart}
\usepackage{bm,graphicx,mathrsfs,amssymb}
\usepackage[foot]{amsaddr}
\usepackage{xy}\xyoption{all}
\usepackage{xcolor}
\usepackage[colorlinks=true]{hyperref}

\newcommand{\mc}{\mathcal}
\newcommand{\mb}{\mathbf}
\newcommand{\bb}{\mathbb}

\DeclareMathOperator{\im}{{\rm im}}

\DeclareMathOperator{\rank}{{\rm rank}}
\DeclareMathOperator{\val}{{\rm val}}

\DeclareMathOperator{\supp}{{\rm supp}}

\newcommand{\conv}{{\rm conv}}

\DeclareMathOperator{\face}{{\rm face}}

\DeclareMathOperator{\lc}{{\rm lc}}
\newcommand{\Gr}{\mathrm{Gr}}
\newcommand{\SI}{\mathrm{SI}}
\newcommand{\bbK}{\bb K}
\newcommand{\bbk}{\mb k} 
\newcommand{\R}{\bb R}
\newcommand{\Rinf}{\bb R_\infty}
\newcommand{\TP}{\bb{TP}}
\newcommand{\TT}{\bb{T}T}
\newcommand{\tc}{{\rm tc}}
\newcommand{\TC}{{\rm TC}}

\newcommand{\cl}{\bm} 
\newcommand{\tmcm}{\mathring{\R}_{\infty}^{d\times n}} 

\let\oldmarginpar\marginpar
\renewcommand\marginpar[1]{\-\oldmarginpar[\footnotesize #1]{\raggedright\footnotesize #1}}

\newtheorem{theorem}{Theorem}[section]
\newtheorem*{theorem*}{Theorem}
  \newtheorem*{theorem_rs}{Theorem \ref{r:restrictsubdiv}}
\newtheorem{lemma}[theorem]{Lemma}
\newtheorem{proposition}[theorem]{Proposition}
\newtheorem{corollary}[theorem]{Corollary}
  \newtheorem*{corollary_si}{Corollary \ref{cor:Stiefel image}}
  \newtheorem*{corollary_rt}{Corollary \ref{r:reg_transversal}}

\theoremstyle{definition}
\newtheorem{definition}[theorem]{Definition}

\theoremstyle{remark}
\newtheorem{remark}[theorem]{Remark}
\newtheorem{example*}[theorem]{Example}

\newenvironment{example}{\begin{example*}\pushQED{\qed}}{\popQED\end{example*}}

\numberwithin{equation}{section}

\begin{document}

\begin{abstract}
The tropical Stiefel map associates to a tropical matrix $A$ its 
tropical Pl\"ucker vector of maximal minors, and thus 
a tropical linear space $L(A)$.  We call the $L(A)$s obtained in this way
\emph{Stiefel tropical linear spaces}.
We prove that they are dual to certain 
matroid subdivisions of polytopes of transversal matroids, 
and we relate their combinatorics
to a canonically associated tropical hyperplane arrangement.
We also explore a broad connection with the secondary fan of the Newton
polytope of the product of all maximal minors of a matrix.
In addition, we investigate the natural parametrization of $L(A)$ arising
from the tropical linear map defined by $A$.
\end{abstract}

\title{Stiefel tropical linear spaces}

\author{Alex Fink$^1$}
\address{$^1$ School of Mathematical Sciences, Queen Mary University of London, United Kingdom.}
\email{$^1$ a.fink@qmul.ac.uk}

\author{Felipe Rinc\'on$^2$}
\address{$^2$ Mathematics Institute, University of Warwick, United Kingdom.}
\email{$^2$ e.f.rincon@warwick.ac.uk}


\hypersetup{
    pdftitle={The tropical Stiefel map},
    pdfauthor={Alex Fink, Felipe Rincon}
    pdfproducer={pdfLaTeX}
}

\dedicatory{Dedicated to the memory of Andrei Zelevinsky.}

\thanks{
The authors thank Bernd Sturmfels for refusing to allow this project to
sink into oblivion, Kristin Shaw for helpful discussions,
and the referees, one of whom corrected a great many infelicities of presentation.
The first author was partially supported by the David and Lucille
Packard Foundation, and the second author by the EPSRC grant EP/I008071/1.
}

\maketitle

\section{Introduction}
Let $d \leq n$ be positive integers. 
In this paper we study a family of tropical linear spaces,
which we call \emph{Stiefel tropical linear spaces},
and their connections to other tropical combinatorial objects
which one may associate to a $d \times n$ tropical matrix.

Any classical $d \times n$ matrix with entries in a field $\mathbb K$
has an associated row space. If the matrix has full rank, 
this row space is $d$-dimensional
and thus yields a point of the Grassmannian $\cl\Gr(d,n)$,
affording the rational \emph{Stiefel map} 
$\mathbb K^{d\times n}\dashrightarrow\cl\Gr(d,n)$.
In tropical geometry, the Grassmannian is tropicalized
with respect to its Pl\"ucker embedding, and it has many of the
properties one might hope for; for instance, it remains a moduli space for
tropicalized linear spaces~\cite{SS}.
Tropicalizing the Stiefel map, one thus gets a map that assigns to each tropical matrix $A$
with entries in $\Rinf := \R \cup \{\infty\}$
a vector $\pi(A)$ in the tropical Grassmannian $\Gr(d,n)$,
namely its vector of tropical maximal minors.
This vector $\pi(A)$ of tropical Pl\"ucker coordinates
is in turn associated to a tropical linear space $L(A)$.
The combinatorial structure of $L(A)$ is determined by 
the regular matroid subdivision induced by $\pi(A)$ \cite{TLS, LTLS}.
We call the tropical linear spaces arising in this way
\emph{Stiefel tropical linear spaces}.

The Stiefel tropical linear space $L(A)$ is the tropicalization of the rowspace
of any sufficiently generic lift of the matrix $A$
to a matrix with entries in $\mathbb{K}$. In this sense, Stiefel tropical linear spaces
arise as tropicalizations of generic linear subspaces of $\mathbb{K}^n$.
Also, as we discuss in Section \ref{sec:tropical}, any Stiefel tropical linear 
space can be thought of as the smallest tropical linear space that ``stably'' contains a 
collection of points. More specifically,
Stiefel tropical linear spaces can be characterized as the tropical linear spaces that are dual to a stable intersection of tropical hyperplanes.

Each of the columns of a tropical matrix $A$ corresponds naturally to a tropical hyperplane in $\Rinf^d$,
so that $A$ determines an arrangement $\mc H(A)$ of $n$ tropical hyperplanes in $\Rinf^d$.
In a similar way, the rows of~$A$ give rise to an arrangement $\mc H(A^{\rm t})$ of $d$ tropical hyperplanes in $\Rinf^n$.
In Section \ref{sec:hyperplanes} we generalize some of the results in \cite{DS, TropOMs} to
show that the combinatorics of these tropical hyperplane arrangements 
are encoded by a regular subdivision $\mc S(A)$ of the root polytope 
$\Gamma_A = \conv \{(e_i,-e_j) : A_{ij} \neq \infty \}$.
Faces in these hyperplane arrangements are encoded by certain bipartite
subgraphs that we call ``tropical covectors'' (also called ``types'' in \cite{DS, TropOMs}),
and are dual to faces of the corresponding mixed subdivisions 
induced by $\mc S(A)$.

In Section \ref{sec:Stiefel TLS combinatorics} we prove an elegant
relationship between the hyperplane arrangement
$\mc H(A^{\rm t})$ and the matroid subdivision dual to $L(A)$.
\begin{theorem_rs}
The regular matroid subdivision $\mc D(A)$ induced by $\pi(A)$
is the restriction to the hypersimplex $\Delta_{d,n}$
of the mixed subdivision dual to $\mc H(A^{\rm t})$.
\end{theorem_rs}
This result, together with certain inequality descriptions for matroid polytopes
of transversal matroids that we give in Section \ref{sec:transversal_polytopes}, 
has the following corollary.
\begin{corollary_rt}
The facets of the regular matroid subdivision $\mc D(A)$
are the matroid polytopes of the transversal matroids associated to
the maximal tropical covectors of the hyperplane arrangement $\mc H(A)$.
\end{corollary_rt}
In this sense, matroid subdivisions corresponding to Stiefel tropical
linear spaces can be thought of as {\em regular transversal matroid subdivisions}.

In general, tropicalizations of algebraic morphisms in the na\"ive sense 
as tuples of polynomial functions over the tropical semiring
are poorly behaved. Their images typically fail
to be tropical varieties or dense subsets thereof, 
let alone tropicalizations of the classical algebro-geometric images.
One of the original motivations for this work
was to find out the extent of this failure for the Stiefel map,
and understand which tropical linear spaces are Stiefel tropical linear spaces.
In the case where $d=2$, where a tropical linear space can be regarded as
a metric tree with $n$ unbounded labelled leaves, the answer can be simply stated:
A tropical linear space in $\Gr(2,n)$ is a Stiefel tropical linear space
if and only if it is a caterpillar tree (see Example \ref{ex:caterpillar}).

As a first step for approaching this question in higher dimensions,
we consider a family of subsets of $[d] \times [n]$
which we call \emph{support sets}. They are introduced in Section~\ref{sec:support sets}.
These subsets have multiple significant interpretations.
For one, they correspond exactly to the minimal graphs whose transversal 
matroid is the uniform matroid $U_{d,n}$ \cite{Bondy};
for another, they index certain significant faces of the 
Newton polytope $\Pi_{d,n}$ of the product of all maximal minors
of a $d\times n$ matrix \cite{SZ}.
If $A \in \Rinf^{d \times n}$, its support is the subset 
$\supp(A) = \{ (i,j) \in [d] \times [n] : A_{ij} \neq \infty \}$.
In Section \ref{sec:tropical} we prove the following result.
\begin{corollary_si}
Every tropical Pl\"ucker vector of the form $\pi(A)$
can be realized in the same form by a matrix $A$ supported
on a support set.
\end{corollary_si}
\noindent Along the way, we prove in Proposition \ref{prop:dim of support face} a conjecture 
of Sturmfels and Zelevinsky stated in \cite[Conjecture 3.8]{SZ},
concerning the dimension of certain distinguished faces of the Newton polytope $\Pi_{d,n}$.
These results are given combinatorial utility
in Section~\ref{sec:Stiefel TLS combinatorics},
as we describe below.

We consider one further combinatorial object associated to~$A$,
first analyzed in \cite{SZ} by way of understanding the
Newton polytope mentioned above:
The matching multifield $\Lambda(A)$ records for each subset
$J \in \binom{[n]}{d}$ the positions where the minimum in the
permutation expansion of the tropical maximal minor with columns $J$ is attained.
We investigate how the
combinatorial structure of the tropical linear space $L(A)$
is related to the matching multifield $\Lambda(A)$.
Tropical combinatorics is acutely sensitive to supports, and
some of our results take their cleanest form when
we restrict our attention to matrices $A$ whose support
is a support set.
In particular, Theorem~\ref{thm:(b) <-> (c)}, Theorem~\ref{thm:(a) -> (c)}, and
Example~\ref{ex:(c) -/> (a)} imply the following result.
\begin{theorem*}
Let $\Sigma$ be a support set.
There is a bijection between combinatorial types of linear spaces $L(A)$ with $\supp(A)=\Sigma$
and coherent matching multifields supported on $\Sigma$,
associating $L(A)$ to $\Lambda(A)$ for each $A$.

Moreover, for matrices $A$ of support $\Sigma$, 
the objects $L(A)$ and $\Lambda(A)$ are determined by $\mc H(A)$,
but this is in general not a bijection.
\end{theorem*}

Finally, in Section \ref{sec:tropical linear maps} we study the tropical linear map $\odot A$
from $\R^d$ to $\R^n$ given by $x \mapsto x \odot A$, in connection to $L(A)$.
The image of this map is a subset of $L(A)$, but unlike the classical case it is in general
a proper subset. In Theorem \ref{r:folding cells} we give a polyhedral description of the tropical linear space $L(A)$
in terms of this map and the hyperplane complex $\mc H(A)$, which expresses $L(A)$
as the union of Minkowski sums of faces of $\im(\odot A)$ with suitable orthants.
Moreover, in Theorem \ref{r:bounded} we prove that the bounded part of $L(A)$ is covered by $\im(\odot A)$,
and we explicitly describe the subcomplex of $\R^d$ it corresponds to.

\subsection{Conventions}
If $P$ is a polyhedron and $u$ a functional on its ambient space,
then $\face_u P$ is the face of $P$ on which $u$ is \emph{minimized}.
If $S$ is a regular subdivision corresponding to the lifted polyhedron 
$\widehat S$, whose faces minimizing the last coordinate project to $S$
on dropping this coordinate, then $\face_u S$ is the projection of $\face_{(u,1)}\widehat S$,
and is called the face of $S$ {\em selected} by $u$.
Normal fans and normal subdivisions to regular subdivisions
are defined with the same conventions: that is, we use inner normal fans.


\section{Matroids and support sets}\label{sec:support sets}
In this section we first introduce some basic
matroidal preliminaries that we will need later in our study. We then define support sets, 
a special class of bipartite graphs that arise naturally in our context, and we recall
some of their main properties from~\cite{SZ}. 

Throughout this paper, we will make constant use of the natural bijection between
bipartite simple graphs on vertex set $[d]\amalg [n]$ and
subsets of $[d]\times[n]$.  We do not differentiate these two
kinds of objects in the notation. As a convention, we reserve the letter $i$
for left vertices of our bipartite graphs (i.e., vertices in $[d]$), and the letter $j$ for right vertices (i.e., vertices in $[n]$). 
The capital letters $I$ and $J$ are reserved for sets of 
objects called $i$ and $j$, respectively.  
In particular, we define the notations for sets of neighbours of a 
left vertex in a bipartite graph $\Sigma \subseteq [d] \times [n]$, or set thereof: 
\begin{align*}
J_i(\Sigma) &= \{j:(i,j)\in\Sigma\}, \\
J_I(\Sigma) &= \bigcup_{i\in I} J_i(\Sigma) = \{j: (i,j)\in\Sigma\mbox{ for some $i\in I$}\},
\end{align*}
and the same for right vertices:
\begin{align*}
I_j(\Sigma) &= \{i:(i,j)\in\Sigma\}, \\
I_J(\Sigma) &= \bigcup_{j\in J} I_j(\Sigma) = \{i: (i,j)\in\Sigma\mbox{ for some $j\in J$}\}.
\end{align*}

\subsection{Matroids}\label{ssec:matroids}
We will assume the reader has a basic knowledge of some of the fundamental notions of matroid theory. A good general reference for this topic is \cite{Oxley}.
Another useful reference is \cite{Murota}, written from a perspective heavier on optimization,
and which goes on to treat valuated matroids (aka tropical Pl\"ucker vectors)
in its section~5.2, prefiguring some tropical results. 

A \emph{partial matching} is a collection of edges $\{(i_1,j_1),\ldots,(i_s,j_s)\} \subseteq [d]\times[n]$ 
such that all the $i_k$ are distinct, as are all the $j_k$.
This partial matching is said to be from 
the set $I = \{i_1,\ldots,i_s\}$ to the set $J = \{j_1,\ldots,j_s\}$, 
or on the set of left vertices $I$ and the set of right vertices $J$.
A \emph{matching} is a maximal partial matching with $[d]$ as its set of left vertices.
In other words, a matching is a set of edges in $[d]\times[n]$ of the form
$\{(1,j_1),\ldots,(d,j_d)\}$, where all the $j_k$ are distinct.
Matchings and partial matchings are at the core of our combinatorial study. 
Matchings appear in matroid theory also under the name \emph{transversals}, but we adopt the graph-theoretic name here.  

Suppose $\Sigma \subseteq [d] \times [n]$ is a bipartite graph on the set of vertices $[d] \amalg [n]$.
The rank $d$ \emph{transversal matroid} $M(\Sigma)$ of this graph 
is the matroid on the ground set $[n]$ 
whose bases are all $d$-subsets $B \subseteq [n]$ for which $\Sigma$ contains a matching
on the set $B$. Note that we are allowing $M(\Sigma)$ to be the matroid with no bases, in the case
that $\Sigma$ contains no matchings. This is not standard practice;
indeed, the matroid with no bases is not usually admitted as a matroid at all.

To any rank $d$ matroid $M$ on ground set~$[n]$ one can associate a
\emph{matroid} (\/\emph{basis}\/) \emph{polytope} \cite{Edmonds, GGMS}
\[\Gamma_M = \conv\Big\{\sum_{j\in B} e_j : \mbox{$B$ is a basis of $M$}\Big\}.\]
This polytope is contained
in the hyperplane $\{x_1+\cdots+x_n=d\}$ of $\R^n$,
and its codimension (in $\R^n$) is equal to the number of connected components of $M$.
If $M$ is the matroid with no bases then $\Gamma_M$ is the empty polytope.

\subsection{Matching fields}
Throughout the paper we will be interested in collections of matchings contained in some bipartite graph $\Sigma$.
\begin{definition}
A {\em matching multifield} $\Lambda$ is a set of matchings
containing at least one matching on each subset $J\in\binom{[n]}d$.
A matching multifield $\Lambda$ is a {\em matching field} if
$\Lambda$ contains a unique matching on each $J\in\binom{[n]}d$.
The {\em support} of a matching (multi)field $\Lambda$ is the union of all the 
edges appearing in some matching in $\Lambda$.
\end{definition}

Let $\Rinf$ be the set $\R\cup\{\infty\}$; in Section~\ref{sec:tropical}
we will see that this is the underlying set of the tropical semifield.
Let $A = (a_{ij}) \in\Rinf^{d\times n}$, and assume that the \emph{support} of~$A$
\[
\supp(A) = \{ (i,j) \in [d] \times [n] : a_{ij} \neq \infty \}
\]
contains at least one matching
on each set~$J \in \binom{[n]}{d}$.  
For such a matrix $A\in\Rinf^{d\times n}$, let
$\Lambda(A)$ denote the matching multifield containing, for
each set of columns $J$, exactly the matchings $\lambda$ on~$J$
which minimize $\sum_{(i,j)\in \lambda} a_{ij}$.
If $A$ is suitably generic then $\Lambda(A)$ will be a matching field.
Using the terminology of Section~\ref{sec:tropical}, the matching multifield $\Lambda(A)$
encodes the positions achieving the minimum in the permutation expansion of each tropical maximal minor of~$A$.

\begin{definition}\label{def:coherent}
The matching multifield $\Lambda$ is {\em coherent} if it arises as 
$\Lambda(A)$ for some matrix $A\in\Rinf^{d\times n}$.  
\end{definition}

We now describe a polyhedral perspective on these notions.
The ($d$\/th) \emph{Birkhoff polytope} is the convex hull
of all permutation matrices in $\R^{d\times d}$,
or equivalently, the Newton polytope of the determinant of a $d\times d$ matrix of indeterminates.
By embedding $\R^{d\times d}$ as the coordinate subspace of
submatrices $\R^{d\times J}\subseteq \R^{d\times n}$
supported on columns $J$, 
we get an image $\Pi_{d,J}$ of the Birkhoff polytope.
A matching on~$J$ is a vertex of $\Pi_{d,J}$.
Taking a matrix $A\in\R^{d\times n}$ to define a linear functional on $\R^{d\times n}$,
a matching multifield $\Lambda$ is coherent if and only if the vertices it selects
of each $\Pi_{d,J}$ are exactly the vertices minimized by this functional.
Let $\Pi_{d,n}$ be the Newton polytope of the product of all maximal minors of
a $d\times n$ matrix, i.e.\ the Minkowski sum of all the $\Pi_{d,J}$.
Vertices of $\Pi_{d,n}$ correspond then to coherent matching fields.
More generally, if $\Lambda$ is a coherent matching multifield 
and $A$ is an associated linear functional,
the face $\face_A \Pi_{d,n}$ uniquely determines $\Lambda$, since 
it determines each of the faces $\face_A \Pi_{d,J}$ of the summands.

The next proposition is a simple generalization of \cite[Proposition~3.1]{SZ} to matching multifields; 
its proof can be obtained following the same arguments, \emph{mutatis mutandis}. Condition (d) is immediate from the definitions in Section~\ref{ssec:matroids}.

\begin{proposition}\label{r:SZ3.1}
If $\Sigma\subseteq[d]\times[n]$, the following are equivalent.
\begin{enumerate}
\renewcommand{\labelenumi}{(\alph{enumi})}
\item $\Sigma$ contains the support of a coherent matching field.
\item $\Sigma$ contains the support of a matching field.
\item For each nonempty $I\subseteq[d]$, 
$|J_I(\Sigma)|\geq n-d+|I|$.
\item The transversal matroid $M(\Sigma)$ is the uniform matroid $U_{d,n}$.
\end{enumerate}
Moreover, ``field'' can be replaced by ``multifield'' in (a) and~(b).
\end{proposition}

Hall's marriage theorem can be stated as solving the problem of determining
when $d$ brides and $d$ grooms can be matched into $d$ marriages,
given the set of bride-groom pairs which are marriageable.  
Postnikov~\cite{PostPAB} extends this to the problem
in which there are $d+1$ brides, any one of which may be stolen away
by a dragon before the marriages are to be made, and states
the necessary and sufficient {\em dragon marriage condition} for
when the marriages are always still possible.  
The equivalence (b)$\Leftrightarrow$(c) of Proposition~\ref{r:SZ3.1} is a 
generalization, which one might call a ``poly-dragon marriage condition'':
now there are $n\geq d$ brides, and any $n-d$ may be stolen by dragons.  

\begin{theorem}[\cite{SZ}, Proposition~3.6]\label{r:SZ3.6}
There exists a (coherent) matching field with support $\Sigma$
if and only if condition~(c) of Proposition~\ref{r:SZ3.1} holds
and equality is achieved when $|I|=1$, i.e.\ 
$|J_i(\Sigma)|=n-d+1$ for each~$i$ (or equivalently, $|\Sigma| = d(n-d+1)$).
\end{theorem}

\begin{definition}
We call a set $\Sigma$ satisfying the equivalent 
conditions of Theorem~\ref{r:SZ3.6} a {\em support set}.  
\end{definition}

The cocircuits of the uniform matroid $U_{d,n}$ are exactly
the subsets of $[n]$ of size $n-d+1$.  Therefore, 
the support sets $\Sigma$ picked out by Theorem~\ref{r:SZ3.6}
are the graphs recognized in \cite[Section~3]{Bondy} as the minimal bipartite graphs
among those whose transversal matroid is $U_{d,n}$,
though their treatment in this context goes back to \cite{LasVergnas,BondyWelsh}.

In the case $n=d+1$ of the usual dragon marriage condition,
we have a convenient graph-theoretical description of the
support sets as a consequence of \cite[Theorem~2.4]{SZ}.

\begin{proposition}\label{r:SZ2.4}
If $n=d+1$, then $\Sigma$ is a support set if and only if,
as a bipartite graph, it is a tree in which every left vertex has degree equal to~$2$.
\end{proposition}

If $\Sigma \subseteq [d] \times [n]$ is any support set,
then there exists a face $\Pi_{d,n}(\Sigma)$ 
of~$\Pi_{d,n}$ whose vertices are exactly
the vertices of~$\Pi_{d,n}$ supported on~$\Sigma$. 
It is the face maximising the linear functional sending a matrix
to the sum of its entries in positions~$\Sigma$ \cite[Proposition~3.7]{SZ}.
The face $\Pi_{d,n}(\Sigma)$
is called a {\em support face}. In fact, it is a consequence of Theorem~\ref{r:SZ3.6} 
that every vertex of $\Pi_{d,n}$ is 
contained in a unique support face, since given a matrix $A\in\R^{d\times n}$
selecting a vertex, the entries not in the support of its matching field may be
replaced by $\infty$.

For later use, we record an immediate consequence of the
discussion following Definition~\ref{def:coherent}.
\begin{proposition}\label{r:bij iii iv}
Suppose $\Sigma$ contains a support set. There is a bijection between 
coherent matching multifields supported
on~$\Sigma$ and faces of~$\Pi_{d,n}(\Sigma)$,
which sends $\Lambda(A)$ to $\face_A \Pi_{d,n}(\Sigma)$.
\end{proposition}

Proposition~\ref{prop:dim of support face} 
ties off the loose end which is Conjecture~3.8 of~\cite{SZ}.
Before proving it, we will need a lemma which will
be helpful for understanding cycles in our bipartite graphs.

\begin{lemma}\label{lem:spanning tree without left leaves}
Suppose $d<n$, and
let $G$ be a connected bipartite graph on the vertex set $[d]\amalg[n]$
such that every edge of $G$ is contained in a matching.
Then $G$ contains a spanning tree with no leaves in the set $[d]$.
Moreover, $M(G)$ is connected.
\end{lemma}

\begin{proof}
Suppose not, and choose a spanning tree $T$ of~$G$ minimizing the 
number of leaves in the set of left vertices.  
We will argue for a contradiction by constructing another
spanning tree $T'$ with fewer left leaves.  

Let $i_0$ be a left vertex that is a leaf of~$T$,
and $j_0$ the right vertex it is adjacent to in~$T$.  Then $j_0$
is incident to at least one other edge of~$T$, and a matching in $G$
containing this edge also contains an edge $f_0$ incident to
$i_0$ other than $(i_0,j_0)$.  There is a single cycle $p_0$
in the graph $G_0 = T\cup\{f_0\}$, and clearly this cycle
includes equally many left and right vertices.

We now iteratively construct
a strictly increasing list of subgraphs $G_k$ of~$G$, $0\leq k\leq m$,
where $G_0$ is the graph described in the previous paragraph.
This list will be finite because the $G_k$ are strictly increasing;
its length will be determined by the construction process.
In fact, the $G_k$ will be unions of successive ears in an \emph{ear decomposition} of~$G$, 
of the sort which is guaranteed to exist by \cite[Theorem 4.1.6]{LovaszPlummer},
with $T$ dictating the choices of ears.

The $G_k$ will be constructed to have the property that,
if $V_k$ is the set of
vertices of $G_k$ contained in some cycle,  
then $V_k$ contains the same number of left and right vertices.  
Since $G_k$ has more
right than left vertices, it contains an edge between $V_k$
and $([n]\amalg[d])\setminus V_k$.
The iterative construction of the graphs $G_k$ 
will stop at the first graph $G_m$
such that there is an edge in $G_m$ having left vertex in 
$V_m$ and right vertex outside $V_m$.  

If there is no such edge for $G_k$, then all edges of~$G$ incident to $V_k$ in just one vertex
are incident to it in a right vertex, and there is at least one such edge.  
Choose a matching of $G$ containing such an edge; because $V_k$ has equally
many left and right vertices, this matching also contains an edge $f_{k+1}$
from a left vertex $i_{k+1}$ of~$V_k$ to a right vertex not in $V_k$.
Let $G_{k+1}$ equal $G_k\cup\{f_{k+1}\}$.  
There is a unique path in $G_k$ from the right endpoint of~$f_{k+1}$ 
to $V_k$; let $p_{k+1}$ be the union of $f_{k+1}$ and this path.
One endpoint of the path $p_{k+1}$ is the left vertex $i_{k+1}$, 
and the other endpoint is a right vertex of $V_k$ by assumption,
so therefore $p_{k+1}$ has equally many 
left as right vertices.
Thus also $V_{k+1} = V_k\cup p_{k+1} = \bigcup_{\ell\leq k+1} p_\ell$ 
contains equally many left as right vertices, as we have claimed.

This iteration finishes with a graph $G_m$
such that some edge of $G_m$ meets $V_m$ in a left vertex $i$
and meets $([n]\amalg[d])\setminus V_m$ in a right vertex $j$.
Observe that every vertex of~$V_m$ is a degree~2 vertex of
exactly one $p_k$ for $k\geq0$.
Let the sequence of indices $k_0,\ldots,k_s$ be defined so that
$i$ is a degree~2 vertex of $p_{k_0}$ and
$i_{k_\ell}$ is a degree~2 vertex of $p_{k_{\ell+1}}$ for all $\ell$
such that $k_\ell>0$; once $k_s=0$ occurs in the sequence,
it terminates.  Let $e_0$ be an edge of $p_{k_0}$ 
incident to $i$, and let $e_{\ell+1}$ be an edge of $p_{k_{\ell+1}}$
incident to $i_{k_\ell}$.
Finally, define
\begin{align*}
T' &= G_m\setminus\{e_0,\ldots,e_s\}\setminus\{f_m : m\neq k_\ell\mbox{ for any $\ell$}\}
\\&= T\cup\{f_{k_0},f_{k_1},\ldots,f_{k_s}=f_0\}\setminus\{e_0,\ldots,e_s\}.
\end{align*}
The graph $T'$ is a spanning tree for $G_m$, and thus for $G$. 
Indeed, its induced subgraph 
on any set $V_k$ is a spanning tree for $V_k$, by an easy induction on~$k$,
and $G_m\setminus V_m$ is a forest with one vertex of each component in~$V_m$,
so that $T'$ as a whole is a spanning tree for $G_m$.
Moreover, $T'$ has one fewer left leaf than $T$ had:
by construction, with two exceptions, the degrees of the left vertices 
in $T'$ are equal to those in $T$, since each incident edge
added in the passage from $T$ to~$G_k$ is balanced by the
removal of an incident edge in the formation of $T'$ from $G_k$.
The exceptions are $i_0$, whose degree has been incremented from 1 to 2,
and $i$, whose degree has been decremented but remains at least 2.  
This is the contradiction sought, and therefore a spanning tree
with no left leaves exists.

Connectedness of $M(G)$ holds for the following reason.
If $M(G)$ is disconnected, then it is the direct sum of 
two matroids $M_1$ and $M_2$ on respective ground sets 
$J$ and $[n]\setminus J$ for some nonempty proper subset $J$ of~$[n]$,
which implies that $|B\cap J| = \rank M_1$ 
for every basis $B$ of~$M(G)$. 
We show that no such nonempty proper subset $J$ exists.  Given $J$,
choose some $j\in J$ and $j'\not\in J$.  Orient the edges
of $T$ away from $j$.  Choose one out-edge
from each left vertex, including all such edges
in the path from $j$ to~$j'$;
such out-edges exist since none of the left vertices
are leaves.  This gives a matching $\lambda$ on~$G$, and a basis $B$ of~$M(G)$.
Now orient the edges away from $j'$; the orientations remain the same
except for those on the path from $j$ to $j'$.  Accordingly
we get another matching $\lambda'$ on~$G$, with corresponding basis
$B\setminus\{j'\}\cup\{j\}$, and these two bases intersect $J$
in sets of different cardinality.
\end{proof}

\begin{corollary}\label{cor:spanning tree without left leaves}
Let $G$ be any bipartite graph on $[d]\amalg[n]$
such that every edge of $G$ is contained in a matching in~$G$.
Then $G$ has a spanning forest $F$
such that for every edge $e$ of $G$ not in~$F$
there exist matchings $\lambda, \lambda'$ on $G$ such that
$e$ is contained in $\lambda$, 
the remaining edges $\lambda\cup\lambda'\setminus\{e\}$ are contained in~$F$,
and the edges contained in just one of $\lambda$ and $\lambda'$
form a single cycle.
\end{corollary}

\begin{proof}
If $d<n$ and $G$ is connected, 
then let $F$ be the tree $T$ provided by Lemma~\ref{lem:spanning tree without left leaves}.
Since $F$ is a spanning tree, the graph $F\cup\{e\}$ contains
exactly one cycle $C$.  Orient the edges of $(F\cup\{e\})\setminus C$
so that they are directed away from~$C$.
Then every left vertex of $G$ not contained in $C$ has
positive outdegree, since none of these vertices are leaves.  
Then we take $\lambda$ to consist
of one out-edge from each of these
vertices, together with alternate edges of $C$ including~$e$,
and $\lambda'$ to consist of the same edges 
from $F\setminus C$ together with the other set of 
alternate edges of $C$, those not including~$e$.
The edges contained in just one of $\lambda$ and $\lambda'$ form~$C$.

If $d=n$ and $G$ is connected, then choose any matching $\lambda'$ on~$G$
and extend it to a spanning tree $F$.  For any edge $e$
of $G\setminus F$, let $\lambda$ consist of the symmetric difference
of $\lambda'$ and the unique cycle of $F\cup\{e\}$.

Finally, if $G$ is disconnected, let $F$ be the union of the spanning
trees for its components provided in the previous paragraphs.
Any partial matching on a component of~$G$ can be extended
to a matching on~$G$ by choosing arbitrary partial matchings
on the other components, and the result follows.
\end{proof}

\begin{proposition}[{\cite[Conjecture 3.8]{SZ}}]\label{prop:dim of support face}
For $d<n$, the dimension of each support face $\Pi_{d,n}(\Sigma)$ equals $(d-1)(n-d-1)$.  
\end{proposition}

\begin{proof}
We claim that the affine span of~$\Pi_{d,n}(\Sigma)$, translated to the origin, 
equals the space spanned by the simplicial 1-cycles of $\Sigma$
orienting all its edges from left to right.  
Since $\Sigma$ is a 1-dimensional complex, this space is $H_1(\Sigma)$,
and its dimension is $h_1(\Sigma)$, the first Betti number.
The graph $\Sigma$ is connected (since if some proper induced subgraph
on vertices $I\subseteq[d]$ and $J\subseteq[n]$ were a connected component,
then at least one of the sets $I$ and $[d]\setminus I$ would violate 
Proposition~\ref{r:SZ3.1}(c)), so $h_1(\Sigma)$ can be computed
from the number of its vertices, which is $d+n$, 
and its edges, which is $d(d-n+1)$ by Theorem~\ref{r:SZ3.6}:
\[h_1(\Sigma)=1-(d+n)+d(d-n+1)=(d-1)(n-d-1).\]

To verify our claim,
let $L$ be the affine span of $\Pi_{d,n}(\Sigma)$.
We have that $\Pi_{d,n}(\Sigma)$ is the face of $\Pi_{d,n}$
obtained by successively minimizing the functionals $x_{i,j}$
for all $(i,j)\not\in\Sigma$.  Minimizing a functional distributes 
across Minkowski sum, so $\Pi_{d,n}(\Sigma)$
is the Minkowski sum over all $J\in\binom{[n]}d$
of the face $\Pi_{d,J}(\Sigma)$
of the Birkhoff polytope $\Pi_{d,J}$ which
minimizes these same functionals.  
This is the face of all points of $\Pi_{d,J}$ whose support is contained in~$\Sigma$. 
The vertices of $\Pi_{d,J}(\Sigma)$
are the matchings on column set $J$ supported on $\Sigma$, and a difference
of the zero-one matrices of two such matchings 
has vanishing boundary, so lies in $H_1(\Sigma)$.
That is, $L\subseteq H_1(\Sigma)$.

Conversely, we apply Corollary~\ref{cor:spanning tree without left leaves}
with $G=\Sigma$.  This yields a spanning tree $T$ of $\Sigma$;
by contracting $T$ to a point we see that
the space $H_1(\Sigma)$ is spanned by the set, 
as $e$ ranges over the edges of $\Sigma$ not in $T$,
of the unique cycles $C$ supported on $T\cup\{e\}$.
For each such edge $e$, the matchings $\lambda$ and $\lambda'$ 
produced by the lemma are on the same column set contained in $\Sigma$,
that is, they are both vertices of one of the polytopes $\Pi_{d,J}(\Sigma)$ above.
As simplicial chains, their difference is the cycle $C$, so that $C\in L$.
This proves $H_1(\Sigma)\subseteq L$.
\end{proof}

Observe that simplicial 1-cycles of the graph $\Sigma$, with the sign choices
given by our orientation, are exactly the signed incidence matrices 
appearing in the discussion preceding Proposition~1.9 of~\cite{SZ}.

\section{Tropical background}\label{sec:tropical}
In this section we introduce material on tropical geometry.
There is at present no canonical general reference for tropical geometry and tropical combinatorics,
but an attempt to rectify this lack is being made by the books \cite{MSbook, J} in preparation.

The tropical semiring is $\mathbb T = (\Rinf, \oplus, \odot)$
where $\Rinf$ is $\R\cup\{\infty\}$, addition $\oplus$ is minimum
(so that $\infty$ is the additive identity),
and multiplication $\odot$ is usual addition.  
Matrix multiplication is defined over a semiring as expected.
The {\em support} of a tropical vector or matrix
is the set of indices of components of that object
which are not equal to $\infty$, that is, which are contained in~$\R$.  

Tropical projective space is $\TP^{k-1}=(\Rinf^k\setminus\{(\infty,\ldots,\infty)\})/\bb R \cdot (1,\ldots,1)$.
The set of points with finite coordinates in tropical projective space,
$\bb R^k/\bb R \cdot (1,\ldots,1)$, is the tropicalization of the big torus,
so we may call it the tropical torus $\TT^{k-1}$.

\subsection{Tropical Grassmannians and linear spaces}
Given an algebraically closed valued field $\bbK$,
let $\bbK^{d\times n}$ denote the variety of $d\times n$ matrices over~$\bbK$. 
Let $\bbK^{d\times n}_{\rm fr}$ denote 
the subvariety of matrices of \underline{f}ull \underline{r}ank, namely rank~$d$.  
The Grassmannian $\cl\Gr(d,n)$ parametrizes $d$-dimensional subspaces of $\bbK^n$, 
and there is a natural map $\cl\pi:\bbK^{d\times n}_{\rm fr}\to\cl\Gr(d,n)$
such that $\cl\pi(\cl A)$ is the space spanned by the rows of~$\cl A$.
We will call this the {\em Stiefel map}.  
The use of Stiefel's name for the map $\cl\pi$ is apparently not usual,
but \cite{GKZ} dubs its domain $\bbK^{d\times n}_{\rm fr}$ the Stiefel variety, 
and the coordinates it provides on~$\cl\Gr(d,n)$ the Stiefel coordinates.

The tropical Grassmannian $\Gr(d,n) \subseteq \TP^{\binom nd-1}$ is the tropicalization of $\cl\Gr(d,n)$ 
via its Pl\"ucker embedding $\cl\iota:\cl\Gr(d,n)\to\bb P^{\binom nd-1}$,
which can be described by saying that 
$\cl\iota \circ \cl\pi$ sends a matrix to its vector of maximal minors.
Any vector $p \in \Gr(d,n)$ satisfies the tropical Pl\"ucker relations:
For any $S, T \subseteq [n]$ such that $|S|=d-1$ and $|T|=d+1$, the minimum
\[
\min_{i \in T \setminus S} (p_{S \cup i} + p_{T - i}) \text{ is achieved at least twice}.
\]
Any vector satisfying these tropical Pl\"ucker relations is called a {\em tropical Pl\"ucker vector} or a {\em valuated matroid} \cite{DW}; we will use the former name here.

\begin{definition}
Let $\tmcm$ be the set of tropical matrices in $\Rinf^{d\times n}$ whose 
support contains a matching.
The {\em tropical Stiefel map} is the map 
$\pi:\tmcm \to\Gr(d,n)$
such that $\pi(A)_J$ is the $([d], J)$ tropical minor of~$A$, that is, if $A = (a_{ij})$ then
\[
\pi(A)_J = \min \Bigg\{ \sum_{(i,j) \in \lambda} a_{ij} : \lambda \text{ is a matching from $[d]$ to } J \Bigg\}.
\]
\end{definition}

Note that the tropical Stiefel map indeed maps matrices in $\tmcm$ to the tropical Grassmannian $\Gr(d,n)$: a sufficiently generic lift of the matrix $A \in \tmcm$ to a matrix $\cl A \in \bbK^{d\times n}_{\rm fr}$ satisfies $\val( \cl\iota \circ \cl \pi (\cl A) ) = \pi(A)$.

Unlike the classical situation, the image of the tropical Stiefel map
is not the whole tropical Grassmannian.  Example~\ref{ex:caterpillar}
describes the simplest case in which they diverge.
\begin{definition}
We call the image in $\Gr(d,n)$ of the tropical Stiefel map the
{\em Stiefel image}, and we denote it by $\SI(d,n)$.
\end{definition}

The tropical Grassmannian $\Gr(d,n)\subseteq\TP^{\binom{[n]}d-1}$
is a polyhedral fan of dimension $d(n-d)$, the same dimension as the
classical Grassmannian $\cl \Gr(d,n)$ 
(as we know in general by the Bieri-Groves theorem \cite{BG}).
We will see in the sequel that $\dim\SI(d,n) = \dim\Gr(d,n) = d(n-d)$.

The torus $(\bbK^\ast)^n$ 
acts on~$\bbK^{d\times n}$ on the right as the diagonal torus in $\mathrm{GL}_n$,
i.e.\ by scaling columns of matrices. This action naturally induces an action of $(\bbK^\ast)^n/\bbK^\ast$
on $\cl \Gr(d,n)$.
The orbits of $(\bbK^\ast)^n/\bbK^\ast$ tropicalize to yield an $(n-1)$-dimensional
lineality space $V$ in $\Gr(d,n)$.
Note that $V$ is also the lineality space of~$\SI(d,n)$.

The tropical Grassmannian $\Gr(d,n)$ is in fact a parameter space
for tropicalized linear spaces~\cite{SS}.
In the classical situation, the linear space associated to a point $\cl p\in\bbK^{\binom{[n]}{n-d}}$
on the Grassmannian $\cl \Gr(d,n)$ is
\begin{equation*}
\cl L(\cl p) = \bigcap_{j_1<\cdots<j_{d+1}\in[n]}
\left\{ \cl y \in \bbK^n : \sum_{k=1}^{d+1} (-1)^k \,
 \cl p_{j_1\cdots\widehat{j_k}\cdots j_{d+1}} \cdot \cl y_{j_k} = 0\right\}.
\end{equation*}
The same holds if we tropicalize all varieties involved~\cite{TLS}, that is,
the tropical linear space with tropical Pl\"ucker vector $p \in \Gr(d,n)$ is
\begin{equation}\label{eq:L(w)}
 L(p) = \bigcap_{J\in \binom{[n]}{d+1}}
\left\{ y \in \Rinf^n : \min_{j \in J} (p_{J - j} + y_j) \text{ is achieved at least twice}\right\}.
\end{equation}
We let $L(A)$ abbreviate $L(\pi(A))$.  This $L(A)$ is the valuated matroid
treated in Example 5.2.3 of~\cite{Murota}.
In the case when all $p_J$ are either $0$ or $\infty$, the space $L(p)$ is also called
the \emph{Bergman fan} of $p$ \cite{AK}.

\begin{definition}
If $L$ is a tropical linear space of the form $L = L(A)$ for some tropical matrix $A$, 
we call $L$ a {\em Stiefel tropical linear space}.
\end{definition}

We say that two tropical linear spaces have the same {\em combinatorial type}
if there is an isomorphism between their posets of faces
sending each face to a face with parallel affine span.
A different description of the combinatorial type of a tropical linear space arises in the context of
matroid polytope subdivisions \cite{murotafamily, TLS, R}, as we explain below.

Let $p \in \Gr(d,n)$, and let $\mc D(p)$ be the regular subdivision of the polytope
$\Gamma_p = \conv \{ e_J : p_J \neq \infty \} \subseteq \R^n$ obtained by projecting back to $\R^n$
the lower faces of the ``lifted polytope'' $\widehat{\Gamma}_p \subseteq \R^{n+1}$ gotten by lifting the vertex $e_J$ to height $p_J$.
The fact that $p$ satisfies the tropical Pl\"ucker relations implies (and is equivalent, in fact)
that the subdivision $\mc D(p)$ is a {\em matroid subdivision}, that is, it is a
subdivision of a matroid polytope into matroid polytopes.
In particular, the collection of subsets $\{ J \in \binom{[n]}{d} : p_J \neq \infty \}$ is the collection
of bases of a rank $d$ matroid over $[n]$, called the {\em underlying matroid} of $p$ (or of $L(p)$).

The part of the tropical linear space $L(p)$ living inside $\R^n$ is a polyhedral complex
dual to the subcomplex of the subdivision $\mc D(p)$
consisting of all those faces of $\mc D(p)$
which are not contained in $\{x_j=0\}$ for any~$j$.
More specifically, for any vector $y \in \R^n$, consider the matroid $M_y$
whose bases are the subsets $J \in \binom{[n]}{d}$ for which
$p_J - \sum_{j \in J} y_j$ is minimal. In other words, the matroid 
polytope of $M_y$ is the projection of the face of $\widehat{\Gamma}_p$
minimized by the functional $(-y, 1)$. We will say that $M_y$ is the matroid
in $\mc D(p)$ {\em selected} by $y$.
It was proved in \cite{TLS, LTLS} that $y \in L(p)$
if and only if $M_y$ is the matroid polytope of a {\em loopless} matroid.
In this perspective, two tropical linear spaces $L(p)$ and $L(p')$ have the same combinatorial type
if and only if their associated matroid subdivisions $\mc D(p)$ and $\mc D(p')$ are equal.

In \cite{TLS}, Speyer described a few operations one can perform on tropical Pl\"ucker vectors
and their corresponding tropical linear spaces. If $p \in \Rinf^{\binom{[n]}{d}}$ is a tropical Pl\"ucker vector,
its {\em dual} $p^* \in \Rinf^{\binom{[n]}{n-d}}$ is defined by 
\[p^*_S = p_{[n] \setminus S}.\]
The {\em stable intersection} of two tropical Pl\"ucker vectors $p \in \Rinf^{\binom{[n]}{d}}$
and $q \in \Rinf^{\binom{[n]}{e}}$ such that $d + e \geq n$ is the tropical Pl\"ucker vector $r \in \Rinf^{\binom{[n]}{d+e-n}}$ given by
\[
r_T = \min_{R \cap S = T} p_R + q_S.
\]
Its corresponding tropical linear space is the stable intersection of the tropical linear spaces associated to $p$ and $q$.

We can dualize this notion as follows: the {\em stable union} of two tropical Pl\"ucker vectors
$p \in \Rinf^{\binom{[n]}{d}}$ and $q \in \Rinf^{\binom{[n]}{e}}$ such that $d + e \leq n$
is the tropical Pl\"ucker vector $r^* \in \Rinf^{\binom{[n]}{d+e}}$ given by 
\[r^*_T = \min_{R \cup S = T} p_R + q_S.\]
Under this terminology, the tropical Stiefel map assigns to a tropical matrix $A$
the tropical Pl\"ucker vector obtained as the stable intersection of its row vectors.
It follows that Stiefel tropical linear spaces are precisely the tropical linear spaces that can
be obtained as the stable union of $d$ points in $\Rinf^n$, as stated in the following proposition.

\begin{proposition}
A tropical linear space $L$ is in the Stiefel image if and only if its dual $L^*$ is a stable intersection of tropical hyperplanes.
\end{proposition}

The operation of stable union of tropical Pl\"ucker vectors is the same as 
the matroid union operation on valuated matroids with the same ground set,
featured in Theorem 5.2.20 of~\cite{Murota}, in the case that rank is additive.

In particular, any Stiefel tropical linear space is \emph{constructible}, 
in the sense discussed in \cite{TLS}.  We expect that constructible tropical
linear spaces are exactly the valuated matroids of Example 5.2.4 of~\cite{Murota}.

\subsection{The tropical meaning of support sets}
From a tropical perspective, 
the culmination of this section is Corollary~\ref{cor:Stiefel image} below,
which shows that the Stiefel image 
is covered by certain polyhedral complexes homeomorphic to real vector spaces.  
In fact, these homeomorphisms are provided, essentially, by restrictions
of the tropical Stiefel map $\pi$ to tropical coordinate subspaces, where
collections of the variables in $A$ are fixed to be $\infty$.
The following definition captures the restrictions of $\pi$ that we use.
\begin{definition}
For a subset $\Sigma\subseteq[d]\times[n]$,
let $\pi_\Sigma$ be the restriction of $\pi$ to the matrices supported on~$\Sigma$.
We will say that $\pi_\Sigma$ is a \emph{supportive restriction} of $\pi$
if its image is a subset of the tropical torus $\TT^{\binom{[n]}{d}-1}$ having dimension $d(n-d)$ and its fibers
are single orbits of the left diagonal $\R^d$ action on tropical matrices
(which acts by adding constants to rows). 
\end{definition}

Note that $d(n-d)$ is the full dimension of the Stiefel image, and
that the fibers of a supportive restriction will in fact be 
isomorphic to $\R^{d-1}$, since this tropical torus $\R^d$ 
induces a faithful action of $\R$ on $\Gr(d,n)$
via the tropical character $x_1+\cdots+x_d$.

\begin{theorem}\label{r:L inj}
The set $\Sigma$ is a support set if and only if $\pi_\Sigma$ is a supportive restriction.
\end{theorem}

For the proof of Theorem \ref{r:L inj} we need the following notion.
A {\em cocircuit} of a tropical linear space $L$ in $\Rinf^n$ is a vector 
$c \in L$ of minimal support $\supp(c) = \{ j \in [n] : c_j \neq \infty \}$.
The following proposition shows that there is a bijection between cocircuits of a tropical linear space and cocircuits of its underlying matroid.

\begin{proposition}[\cite{MT}]
Let $L$ be a tropical linear space. The support of any cocircuit $c$ of $L$ is a cocircuit of the underlying matroid of $L$.
Moreover, if two cocircuits of $L$ have the same support then they differ by a scalar multiple 
of the vector $(1,\dotsc,1) \in \bb R^n$, that is, they are equal in $\TP^{n-1}$.
\end{proposition}

In fact, any tropical linear space $L$ is equal to the tropical convex hull of its cocircuits \cite{MT, YY, R},
so cocircuits can be thought of as vertices of $L$ from a tropical convexity point of view.

\begin{proof}[Proof of Theorem~\ref{r:L inj}.]
Suppose $\pi_\Sigma$ is a supportive restriction.  
The coordinates of $\pi_\Sigma$ are tropical maximal minors.  
None of these may be identically $\infty$, 
so $\Sigma$ supports a matching on every $J\in\binom{[n]}d$,
i.e.\ condition~(b) of~Proposition~\ref{r:SZ3.1} holds. 
The dimension of the domain of~$\pi_\Sigma$ 
is the sum of the dimensions of the fibers and the image,
which is $d + d(n-d) = d(n-d+1)$.  But this dimension is $|\Sigma|$,
and condition~(c) of Proposition~\ref{r:SZ3.1} 
applied to each singleton set $I$
implies that $|\Sigma|\leq d(n-d+1)$.
The equality case is achieved only when $|J_i(\Sigma)|=n-d+1$ for each~$i$.
Hence $\Sigma$ is definitionally a support set.

Now, assume $\Sigma$ is a support set. We want to prove that the map $A \mapsto L(A)$ is injective for
 matrices with support $\Sigma$, up to tropical rescaling of the
 rows. For this purpose we describe how to recover the
matrix $A$ from $\Sigma$ and the tropical linear space $L(A)$.
 Since $\Sigma$ is a support set, the underlying matroid of $L(A)$ 
is simply the uniform matroid $U_{d,n}$. The cocircuits of $L(A)$ are then
the vectors $c \in L(A)$ whose support is in $\binom{[n]}{n-d+1}$. 
In view of Theorem \ref{r:SZ3.6}, the rows of $A$ are cocircuits of $L(A)$, 
so they correspond precisely to the cocircuits
of $L(A)$ whose support is equal to one of the sets $J_i(\Sigma)$.
\end{proof}

Suppose $A\in \Rinf^{d\times n}$ is a matrix whose matching multifield
$\Lambda(A)$ is in fact a matching field.
By replacing the entries outside the support $\Sigma$ of $\Lambda(A)$
by $\infty$ yields a matrix with support $\Sigma$
and the same tropical minors as~$A$.
Any matrix whose support contains a matching multifield is a limit
of such matrices $A$. Since each $\pi_\Sigma$ is continuous, the following result follows.
\begin{corollary}\label{cor:Stiefel image}
The part of the Stiefel image $\SI(d,n)$ inside the tropical torus $\TT^{\binom{[n]}{d}-1}$ is equal
to the union of the images $\im\pi_\Sigma$ over all support sets $\Sigma$.
\end{corollary}
In view of Theorem \ref{r:L inj}, Corollary \ref{cor:Stiefel image} describes
the finite part of $\SI(d,n)$ as a union of nicely parametrized sets that are homeomorphic to real vector spaces.

\begin{example}\label{ex:pointed}
The {\em pointed support sets} are a class of support sets
for any $d \leq n$.
Up to a reordering of the set of columns $[n]$, they have the form
\[\Sigma = \{(i,i):i\in[d]\}\cup\{(i,j):i\in[n]\setminus[d] \text{ and } j\in[d]\},\]
corresponding to matrices of the form
\[\begin{pmatrix}
\ast & \infty & \cdots & \infty & \infty & \ast & \ast & \cdots & \ast & \ast \\
\infty & \ast & \cdots & \infty & \infty & \ast & \ast & \cdots & \ast & \ast \\
\vdots & \vdots & \ddots & \vdots & \vdots & \vdots & \vdots & \ddots & \vdots & \vdots \\
\infty & \infty & \cdots & \ast & \infty & \ast & \ast & \cdots & \ast & \ast \\
\infty & \infty & \cdots & \infty & \ast & \ast & \ast & \cdots & \ast & \ast \\
\end{pmatrix}\]
where each $\ast$ represents a real number. Tropical linear spaces in the image of $\pi_\Sigma$
for $\Sigma$ a pointed support set have been studied in \cite{HJS, LTLS}, where it was shown
that their dual matroid subdivisions are {\em conical matroid subdivisions}, i.e. all
their facets share a common vertex. This fact was used in \cite{LTLS} to give a
simple proof that these tropical linear spaces satisfy Speyer's $f$-vector conjecture.
\end{example}

\begin{example}\label{ex:caterpillar}
Let $d=2$. Tropical linear spaces are then metric trees with
a single unbounded edge in each coordinate direction, which we regard
as labelled leaves.
It is easy to see from Theorem \ref{r:SZ3.6} that, in this case,
any support set is in fact a pointed support set.
It was shown in \cite[Example 4.2]{LTLS}, and it also follows from Theorem~\ref{r:bounded} below,
that any tropical linear space in the part of the Stiefel image $\SI(2,n)$
inside the tropical torus $\TT^{\binom{[n]}{2}-1}$
must then be a {\em caterpillar tree}, i.e.\ 
a tree obtained by gluing rays to a homeomorphic image of $\R$,
whose bounded part must therefore be homeomorphic to a segment.   
In particular, the {\em snowflake tree}
of Figure~\ref{f:snowflake} is the combinatorial type of a tropical linear space
in $\Gr(2,6)$ which is not in the Stiefel image $\SI(2,6)$.
\end{example}

\begin{figure}
\centering
\includegraphics{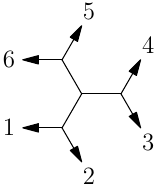}
\caption{A \emph{snowflake tree}, representing the smallest tropical linear
space which is not contained in the Stiefel image.}
\label{f:snowflake}
\end{figure}

\section{Tropical hyperplane arrangements}\label{sec:hyperplanes}
In this section we generalize to arbitrary matrices in $\Rinf^{d \times n}$ 
the results in \cite{DS, TropOMs} relating the combinatorics 
of a tropical hyperplane arrangement with an appropriate regular subdivision of 
a product of simplices. We then use this machinery to investigate the connection between
the matching multifield of a tropical matrix and the combinatorial type of its associated
tropical hyperplane arrangement.

For any positive integer $m$ and any $K \subseteq [m]$, consider the simplex
\[\Delta_K = \conv\{e_k : k\in K\} \subseteq \R^m. \]
The faces of the standard $(m-1)$-dimensional simplex $\Delta_{[m]}$
are naturally in bijection with subsets $K$ of $[m]$, with $K$ associated to $\Delta_K$.

Given a point $a \in \TP^{m-1}$, the \emph{tropical hyperplane} $H \subseteq \TP^{m-1}$ with
vertex $-a$ is the set
\begin{equation}\label{eq:hyperplane}
H = \left\{ x \in \TP^{m-1} : \min_{k \in [m]} (a_k + x_k) \text{ is achieved at least twice} \right\}.
\end{equation}
Let $K$ denote the support $\supp(a) = \{ k \in [m] : a_k \neq \infty \}$ of $a$.
We will also say that $K$ is the {\em support} of~$H$.
The part of $H$ inside $\TT^{m-1}$ is
the codimension $1$ skeleton of a translate of the normal
fan of the simplex $\Delta_K$, so it naturally comes endowed with a fan structure. 
The faces of $H$ are in bijection with subsets of~$K$ of size at least 2,
corresponding to the positions where the minimum \eqref{eq:hyperplane} defining $H$ is attained.
We will find it useful, however, to consider the complete fan induced by $H$ on $\TT^{m-1}$:
if $L \subseteq K$ is nonempty, we will write 
\[
F_L(H) = \left\{ x \in \TT^{m-1} : a_l + x_l = \min_{k \in [m]} (a_k + x_k) \text{ for all } l \in L \right\}.
\]
In the case $L = \{\ell\}$, we will simply write $F_\ell(H)$ for the sector $F_{\{\ell\}}(H)$. We have $F_L(H) = \bigcap_{\ell\in L} F_{\ell}(H)$.
Figure~\ref{f:side sets} depicts some tropical hyperplanes in $\TT^2$, together with their corresponding supports.

\begin{figure}[ht]
\centering
\includegraphics[width=12cm]{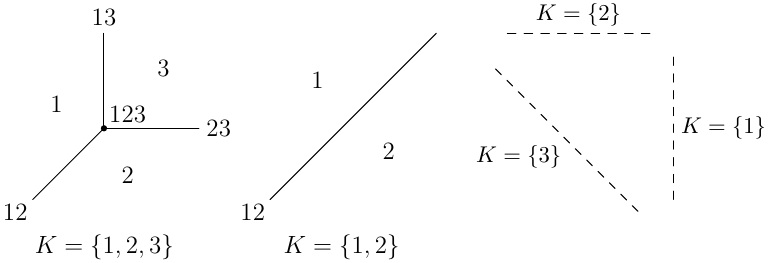}
\caption{At the left and center, two tropical hyperplanes in $\TT^2$,
with the faces of their induced complete fans labelled with subsets of their corresponding support~$K$.
At the right, a more schematic illustration of three tropical lines 
in $\TP^2$ residing at infinity.
}
\label{f:side sets}
\end{figure}

A matrix $A=(a_{ij})\in\Rinf^{d\times n}$ in which no column 
has all entries equal to $\infty$ gives rise to a (ordered) tropical
hyperplane arrangement $\mc H(A) = (H_1,\ldots,H_n)$ 
in~$\TP^{d-1}$ whose $j$\/th hyperplane is 
the tropical hyperplane with vertex $(-a_{ij})_{i \in [d]}$.
The {\em support} of a hyperplane arrangement $\mc H=(H_1,\ldots,H_n)$
is the set
\[
\supp(\mc H) = \{(i,j) \in [d] \times [n]: \text{$i$ is in the support of $H_j$}\},
\]
so that the support of $\mc H(A)$ is equal to the support of the matrix~$A$.
All these notions make sense also in the case where $d > n$,
and we shall consider later the tropical arrangement $\mc H(A^{\rm t})$ in $\TP^{n-1}$
determined by the rows of $A$.

We will find it convenient to identify the tropical hyperplane arrangement
$\mc H$ with the (labelled) polyhedral complex supported on $\TT^{d-1}$
which is the common refinement of the associated complete fans, as we describe below.  
Each point $x\in\TT^{d-1}$ determines a ({\em tropical})
{\em covector}\footnote{For what we call a tropical covector 
the term used in the literature \cite{TropOMs,DS} is \emph{type}.  We find
the semantically meager word ``type'' overburdened with
specific senses in mathematics, and wish to avoid increasing its load.
The contrast with ``combinatorial type'' is especially unfortunate.}
\[\tc(x) = \big\{(i,j)\in[d]\times[n] : x\in F_{j}(H_i)\,\big\}.\]
In the interest of making our notation not too cumbersome, 
we will sometimes describe a covector $\tau \subseteq [d] \times [n]$
by the tuple $(I_1(\tau), I_2(\tau), \dotsc , I_n(\tau))$.
The faces of $\mc H$ are then the closures of the
domains within $\TT^{d-1}$
on which the covector is constant.
By the {\em combinatorial type} of~$\mc H$ we mean the collection of
the covectors of all its faces.
We will write $\TC(\mc H)$ for this set of covectors\footnote{$\TC(\mc H)$
being a set of objects named $\tc$, 
or standing for \emph{type combinatoire} and \emph{tipo combinatorio}.},
and let $\TC(A)$ abbreviate $\TC(\mc H(A))$.
Figure~\ref{f:hyparr} depicts a tropical hyperplane arrangement in $\TP^2$,
together with the tropical covectors labelling some of its faces.

\begin{figure}[ht]
\centering
\begin{minipage}[c]{1.7in}
\[A = \begin{pmatrix}
0 & 3 & 0 & \infty & \infty \\
\infty & 0 & 0 & 2 & \infty \\
\infty & \infty & 0 & 0 & 0
\end{pmatrix}\]
\end{minipage}
\qquad
\raisebox{-0.5\height}{\includegraphics{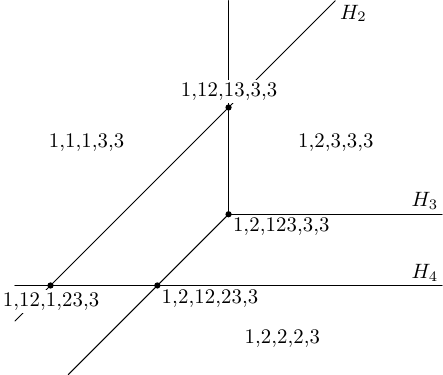}}
\caption{A matrix $A$, supported on a support set, and the tropical hyperplane
arrangement it induces.  The vertices and a few other faces of the arrangement are labelled
with their tropical covectors.  Hyperplanes $H_1$ and~$H_5$, at infinity, are not drawn.}
\label{f:hyparr}
\end{figure}

Much previous work on the combinatorics of tropical hyperplane arrangements 
has restricted itself to the case of arrangements  
with the full support $[d] \times [n]$ (called nondegenerate arrangements),
 in which all of the hyperplanes are of 
the shape containing a vertex.
Such arrangements, for instance, were the subject of Develin and Sturmfels in~\cite{DS}.
The combinatorial type of such an arrangement
$\mc H$ is the data of the tropical oriented matroid 
which Ardila and Develin associate to~$\mc H$ in~\cite{TropOMs}.
We will prefer here to look at arrangements $\mc H$ with arbitrary support sets.

A duality between nondegenerate hyperplane arrangements and subdivisions of 
a product of simplices was described in~\cite{DS}. We now show how this duality 
can be generalized to arrangements with more general support sets.
Suppose $A = (a_{ij}) \in \Rinf^{d\times n}$ is a tropical matrix
in which no column has all entries equal to $\infty$.
Let $\Sigma$ be the support of $A$, and let $\Gamma_\Sigma$ be the polytope
\[
 \Gamma_\Sigma = \conv \{ (e_i,-e_j) \in \R^d \times \R^n : (i,j) \in \Sigma \}.
\]
Polytopes of the form $\Gamma_\Sigma$ are called {\em root polytopes} in \cite{PostPAB},
where Postnikov studies them in connection to generalized permutohedra.
The matrix $A$ induces a regular subdivision $\mc S(A)$ 
of $\Gamma_\Sigma$ by lifting the vertex 
$(e_i,-e_j)$ to height $a_{ij}$ and projecting back to $\R^d \times \R^n$
the lower faces of the resulting polytope. 
As described in \cite[Section 14]{PostPAB}, the Cayley trick allows us
to encode the subdivision $\mc S(A)$ using a 
mixed subdivision $\mc M(A)$ of the sum of simplices 
$\sum_{j=1}^n \Delta_{I_j(\Sigma)} \subseteq \R^d$, 
where $I_j(\Sigma) = \{ i \in [d] : (i,j) \in \Sigma \}$. This mixed subdivision
is isomorphic to the subcomplex $\mc T(A)$ of $\mc S(A)$ consisting of 
the faces that contain for every $j \in [n]$ at least one vertex of the form $(e_i,-e_j)$.
The subcomplex $\mc T(A) \cong \mc M(A)$ determines the whole subdivision $\mc S(A)$ (and vice versa); see \cite[Proposition 14.5]{PostPAB}.

\begin{proposition}\label{r:dual mixed}
 The hyperplane complex $\mc H(A)$ is dual to the mixed subdivision $\mc M(A)$ of $\sum_{j=1}^n \Delta_{I_j(\Sigma)}$.
 A face of $\mc H(A)$ labelled by a tropical covector $\tau \subseteq \Sigma$ is dual to the cell
 of $\mc M(A)$ obtained as the sum $\sum_{j=1}^n \Delta_{I_j(\tau)}$.
\end{proposition}

Figure~\ref{f:mixed_subdiv} shows the mixed subdivision dual to the hyperplane arrangement presented in Figure~\ref{f:hyparr}.
\begin{figure}[ht]
\centering
\includegraphics[scale=0.3]{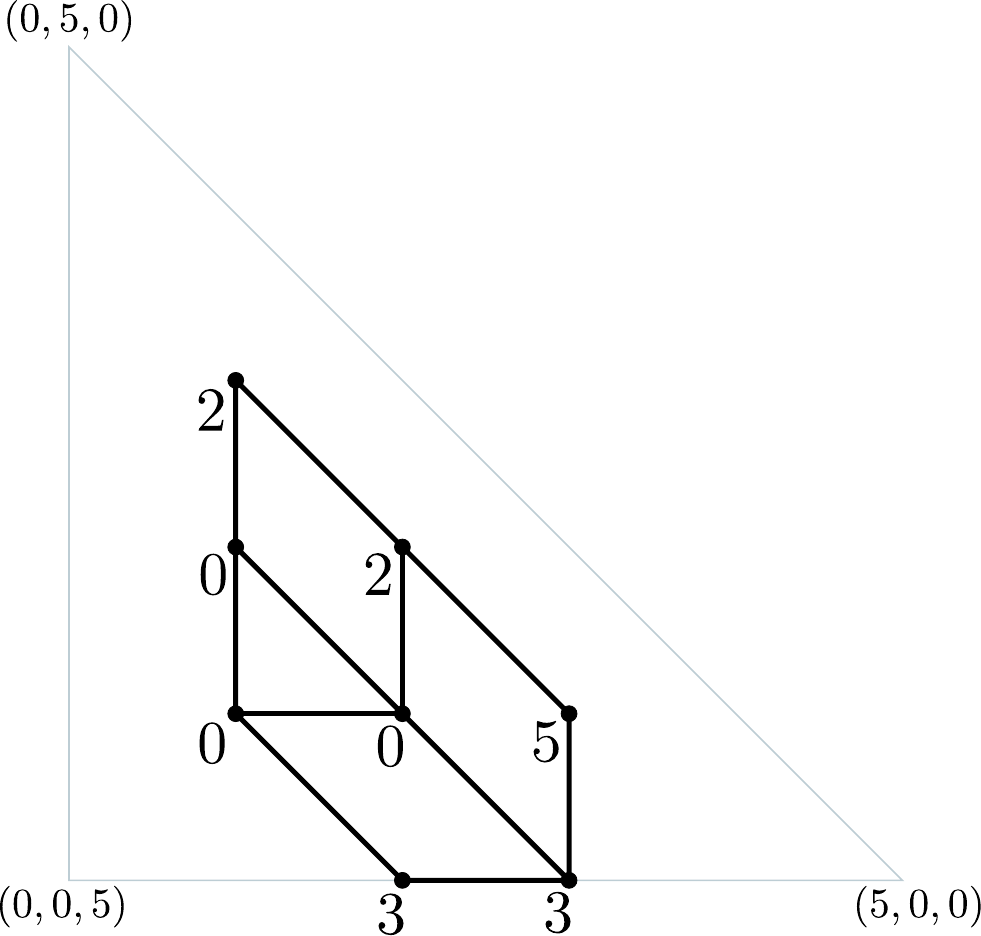}
\caption{The mixed subdivision $\mc M(A)$ dual to the (negative of the) hyperplane arrangement in
Figure~\ref{f:hyparr}, sitting inside the simplex $5 \cdot \Delta_{[3]}$.
The numbers represent lifting heights for the
regular subdivision.}
\label{f:mixed_subdiv}
\end{figure}

Proposition \ref{r:dual mixed} is a direct generalization of Lemma~22 
in~\cite{DS}, but we will use later some of the ideas involved in its proof, in connection
to tropical linear spaces. 
We give a proof for explicitness.

\begin{proof}[Proof of Proposition \ref{r:dual mixed}] 
 Consider the polyhedron 
\begin{equation}
 \mc P_A = \{ (x,y) \in \R^d \times \R^n : x_i + a_{ij} \geq y_j \text{ for all } (i,j) \in \Sigma \}.
\end{equation}
The facets of $\mc P_A$ are given by the hyperplanes $h_{ij} = \{ (x,y) \in \R^d \times \R^n : x_i + a_{ij} = y_j \}$,
with $(i,j) \in \Sigma$.
A nonempty subcollection of these hyperplanes specifies a (nonempty)
face of $\mc P_A$ if and only if the corresponding vertices 
of $\Gamma_\Sigma$ form a face in the subdivision $\mc S(A)$,
so the boundary $\partial \mc P_A$ of $\mc P_A$ is dual to $\mc S(A)$. Denote by $\mc Q(A)$ the
subcomplex of $\partial \mc P_A$ consisting of the faces contained for 
every $j \in [n]$ in at least one of the hyperplanes $h_{ij}$.
The complex $\mc Q(A)$ is then dual to the subcomplex $\mc T(A)$ of $\mc S(A)$ and thus to the mixed subdivision $\mc M(A)$.

Denote by $\phi: \R^d \times \R^n \to \R^d$ the projection onto the first factor. 
We will prove that $\phi$ induces an isomorphism between the complexes $\mc Q(A)$ and $\mc H(A)$,
thus proving our first assertion. 
Given $x \in \R^d$, for any point $(x,y)\in \mc P_A$ 
we have $y_j \leq x_i + a_{ij}$ for all $(i,j) \in \Sigma$.  
If also $(x,y) \in \mc Q(A)$ then for every~$j\in[n]$ there is some
$i\in[d]$ such that $y_j = x_i + a_{ij}$.  In this situation
$y_j$ is uniquely determined to be~$\min_{i\in[d]}x_i+a_{ij}$
for each~$j$, or more elegantly, $y = x \odot A$
(where $\odot$ denotes tropical matrix multiplication).
This shows that $\phi|_{\mc Q(A)}$ is injective.
On the other hand, for any $x \in \R^d$ the point $(x, x \odot A)\in \R^d \times \R^n$
is in $\mc Q(A)$, so $\phi|_{\mc Q(A)}$ is also surjective.
Moreover, a pair $(i,j)$ is in the covector $\tc(x)$ 
if and only if the point $(x, x \odot A)$ is in the hyperplane $h_{ij}$,
showing that $\phi|_{\mc Q(A)}$ preserves the polyhedral complex structure.

The exact correspondence between the faces follows from tracking faces in the argument above.
A cell $\sum_{j=1}^n \Delta_{I_j(\tau)}$ of $\mc M(A)$ corresponds to the face
$\conv\{(e_i,-e_j) : (i,j) \in \tau \}$ of $\mc S(A)$. This face in turn corresponds to
the face $\bigcap_{(i,j) \in \tau} h_{ij}$ of $\mc P_A$, which after projecting back to $\R^n$
gets mapped to the face of $\mc H(A)$ labelled by~$\tau$.
\end{proof}

\begin{corollary}\label{r:bijection_arr_sub}
There is a bijection between combinatorial types of arrangements
of $n$ tropical hyperplanes in $\TP^{d-1}$
and regular subdivisions of subpolytopes $P$ of the product of
simplices $\Gamma_{[d] \times [n]} \cong \Delta_{[d]} \times \Delta_{[n]}$, such that 
$P$ has a vertex of the form $(e_i,-e_j)$ for each $j \in [n]$.
\end{corollary}

\begin{remark}\label{rem:transpose_arr}
Note that Corollary \ref{r:bijection_arr_sub} implies a certain duality between
arrangements of $n$ tropical hyperplanes in $\TP^{d-1}$ and
arrangements of $d$ tropical hyperplanes in $\TP^{n-1}$. In fact,
suppose that $A \in \Rinf^{d\times n}$ is a tropical matrix
in which no column and no row has all entries equal to $\infty$.
Covectors labelling the faces of the hyperplane arrangement $\mc H(A)$
correspond to faces of $\mc S(A)$ that contain for every $j \in [n]$
at least one vertex of the form $(e_i,-e_j)$. Since $\mc S(A^{\rm t})$ is naturally
isomorphic to $\mc S(A)$, we also have that
covectors labelling the faces of $\mc H(A^{\rm t})$
correspond to faces of $\mc S(A)$ that contain for every $i \in [d]$
at least one vertex of the form $(e_i,-e_j)$.
In particular, the sets of covectors appearing in $\mc H(A)$ and $\mc H(A^{\rm t})$
in which all vertices have degree at least 1 are exactly the same
(up to exchanging the roles of $d$ and $n$).
Moreover, this set of covectors includes covectors for all interior faces of the subdivision $\mc S(A)$ and thus
it completely determines $\mc S(A)$, together with
the combinatorics of both $\mc H(A)$ and $\mc H(A^{\rm t})$.
We will push these ideas further in Proposition \ref{r:odot A}.
\end{remark}

We now turn to the study of the connection between the matching multifield $\Lambda(A)$
and the combinatorial type of the hyperplane arrangement $\mc H(A)$. 
For the rest of this section we will restrict our attention to matrices $A$ whose
support contains a matching field.

A tropical square matrix is called {\em tropically singular} if the minimum in
the permutation expansion of its tropical determinant is achieved at least twice.
The {\em tropical rank} of a matrix $A \in \Rinf^{d \times n}$ is the largest $r$ such that
$A$ contains a tropically non-singular $r \times r$ minor.

\begin{theorem}\label{thm:(a) -> (c)}
The matching multifield $\Lambda(A)$ depends only on the combinatorial type $\TC(A)$ of the hyperplane
arrangement $\mc H(A)$.  That is, if $A, A'\in\Rinf^{d \times n}$ are matrices
such that $\TC(A)=\TC(A')$, then $\Lambda(A) = \Lambda(A')$.
\end{theorem}

\begin{proof}
By the Cayley trick as manifested in Corollary~\ref{r:bijection_arr_sub}, 
it is enough to prove that the cell complex $\mc S(A)$ determines $\Lambda(A)$.
Let $\Sigma$ denote the support of $A$.
Faces in the subdivision $\mc S(A)$ correspond to subsets $\tau \subseteq \Sigma$
for which it is possible to add constants to the rows or columns of $A$ to get a
non-negative matrix $A' = (a'_{ij})$ satisfying $a'_{ij} = 0$ if and only if $(i,j)\in \tau$.
The matching multifield $\Lambda(A')$ determined by such a matrix $A'$ is always equal to $\Lambda(A)$.
Moreover, if $\tau$ contains some matching $\{(1,j_1), \dotsc, (d,j_d)\}$ then the set of
matchings on the set $J = \{j_1,\dotsc,j_d\}$ in the multifield $\Lambda(A')$ is precisely
the set of matchings on $J$ contained in $\tau$. 

We claim that in fact any matching
in $\Lambda(A)$ is contained in some subset $\tau$ corresponding to a face in $\mc S(A)$.
To see this, assume that $\lambda$ is a matching on the set $J$, and $\lambda \in \Lambda(A)$.
We can perturb the matrix $A$ a little to get a matrix $B$ in which $\lambda$ is the only
matching on the set $J$ included in $\Lambda(B)$. This means that the square submatrix $B_J$ of $B$
indexed by the columns $J$ has tropical rank equal to $d$, and so 
\cite[Corollary 5.4]{DSS} implies that the mixed subdivision $\mc M(B_J)$ has an interior vertex. The covector indexing this vertex must be a matching on $J$, and thus equal to $\lambda$.
The matching $\lambda$ also indexes a face of the subdivision $\mc S(B)$, as can be seen by adding 
large enough constants to the columns of $B$ not in $J$. Finally, since the subdivision $\mc S(B)$
is a refinement of $\mc S(A)$, the matching $\lambda$ is contained 
in some subset $\tau$ corresponding to a face in $\mc S(A)$.
It follows that the matching multifield $\Lambda(A)$ is precisely the set of matchings
contained in some subset $\tau$ corresponding to a face in $\mc S(A)$.
\end{proof}

\begin{remark}
It follows from the proof of Theorem \ref{thm:(a) -> (c)} 
that the matching multifield $\Lambda(A)$
is equal to the set of matchings contained in subsets $\tau \subseteq \Sigma$ 
corresponding to interior faces of $\mc S(A)$.
Following the ideas given in Remark \ref{rem:transpose_arr},
we thus have that $\Lambda(A)$ is also equal to 
the set of matchings contained in some covector appearing in $\TC(A)$.
Similarly, $\Lambda(A)$ is equal to 
the set of matchings contained in some covector of $\TC(A^{\rm t})$
(after exchanging the roles of $d$ and $n$).
\end{remark}

To say more about the connection stated by Theorem~\ref{thm:(a) -> (c)},
we will analyze it in terms of the fan structures induced on
the set $\R^\Sigma\subseteq\Rinf^{d\times n}$ of matrices whose
support is $\Sigma$ governing the combinatorics of
$\mc H(A)$ and $\Lambda(A)$, respectively.
Assume $A \in \Rinf^{d \times n}$ has support $\Sigma$. 
As stated in Proposition \ref{r:dual mixed}, the combinatorial type
of the hyperplane arrangement $\mc H(A)$ is encoded by the regular
subdivision $\mc S(A)$ induced by $A$ on the polytope $\Gamma_\Sigma$,
or what is equivalent, by which cone of the secondary fan of $\Gamma_\Sigma$
contains $A$.  
On the other hand, Proposition \ref{r:bij iii iv} states that 
the matching multifield $\Lambda(A)$ is encoded by
the data of which face is $\face_A \Pi_{d,n}(\Sigma)$, that is, 
by which cone of the normal fan of $\Pi_{d,n}(\Sigma)$
contains $A$.

\pagebreak[2] 
\begin{proposition}\label{prop:codim 1 skeleta}
Let $\Sigma$ contain a support set.

\begin{enumerate}\renewcommand{\labelenumi}{(\roman{enumi})}
\item The support of the codimension~1 skeleton of
the secondary fan of $\Gamma_\Sigma$ is
the set of matrices $A$ supported on $\Sigma$ such that 
the minimum in some non-infinite tropical minor of $A$ is attained twice.

\item The support of the codimension~1 skeleton of
the normal fan of $\Pi_{d,n}(\Sigma)$ is 
the set of matrices $A$ supported on $\Sigma$ such that 
the minimum in some (non-infinite) tropical \emph{maximal} minor of $A$ is attained twice.
\end{enumerate}
\end{proposition}

If $\lambda = \{(i_1,j_1),\ldots,(i_s,j_s)\}$ is a partial matching, we denote
\[
A_\lambda = a_{i_1j_1} + \cdots + a_{i_sj_s}.
\]
Recall that the minimum of the $(I,J)$ tropical maximal minor is
not infinite and is attained twice if and only if there are
two distinct partial matchings $\lambda,\lambda'$ from $I$ to~$J$ contained in $\Sigma$
such that, for any such partial matching $\lambda''$,
it holds that $A_\lambda = A_{\lambda'} \leq A_{\lambda''}$.

Example~\ref{ex:(c) -/> (a)} is an example of a support set $\Sigma$ where the two fans described by Proposition \ref{prop:codim 1 skeleta} differ.
Example~\ref{ex:(c) -> (a) pointed} is one in which they do not.

\begin{example}\label{ex:(c) -/> (a)}
The following matrix $A(t)$ is supported on a support set, for real $t$:
\[
A(t) = \begin{pmatrix}
\infty & 0 & 0 & 0 & \infty & \infty \\
0 & \infty & 0 & \infty & 0 & \infty \\
t & 0 & \infty & \infty & \infty & 0 \\
0 & 1 & 2 & \infty & \infty & \infty
\end{pmatrix}.
\]
If $|t|<1$ then the matching multifield $\Lambda(A(t))$ is independent of~$t$;
the minimum matching on column set $J$ is always the unique one
which chooses the least entry possible in the fourth row.
However, the hyperplane arrangement $\mc H(A(t))$ undergoes
a change of combinatorial type when $t$ passes from positive to negative.
There exist points of covector $(2,3,1,1,2,3)$ only when $t>0$,
and points of covector $(3,1,2,1,2,3)$ only when $t<0$. 
In view of Proposition \ref{prop:codim 1 skeleta}, the reason why 
this happens is the existence of the matchings $\lambda = \{ (1,2), (2,3), (3,1) \}$
and $\lambda' = \{ (1,3), (2,1), (3,2) \}$ satisfying 
$A(0)_{\lambda} = A(0)_{\lambda'} \leq A(0)_{\lambda''}$ for any partial matching
$\lambda''$ on the same row and column sets, but which cannot be extended to matchings
in $\Sigma$ satisfying a similar condition. In fact, $\lambda$ and $\lambda'$ cannot 
be extended to any matching in $\Sigma$.
\end{example}

\begin{example}\label{ex:(c) -> (a) pointed}
Let $\Sigma$ be the pointed support set (see Example \ref{ex:pointed}).  If $A$ has support $\Sigma$
and the $(I,J)$ tropical minor of $A$ is not infinite, then $J$ is 
disjoint from $[d]\setminus I$ (under the identification of
the left vertices $[d]$ with a subset of the right vertices $[n]$),
otherwise a column of the $(I,J)$ submatrix would contain only infinities.
So $([d],J\cup[d]\setminus I)$ indexes a tropical maximal minor of~$A$.
The only matchings on column set $J\cup[d]\setminus I$ supported on~$A$
select the $(j,j)$ entry of each column $[d]\setminus I$,
and therefore each term in the permutation expansion of the 
$([d],J\cup[d]\setminus I)$ tropical maximal minor is a constant
plus a term of the permutation expansion of the $(I,J)$ tropical minor.
It follows that conditions (i) and (ii) of Proposition~\ref{prop:codim 1 skeleta} 
are equivalent in this case. We thus have that the injection described in 
Theorem \ref{thm:(a) -> (c)} is in fact a bijection: there is a correspondence
between coherent matching multifields supported on $\Sigma$ and 
regular subdivisions of the root polytope $\Gamma_{\Sigma}$.
This corresponds to the fact that the Newton polytope of the product
of all minors of a $d \times (n-d)$ matrix is the secondary polytope
of the product of simplices $\Delta^{d-1} \times \Delta^{n-d-1}$ (see \cite{GKZ}).
\end{example}

\begin{proof}[Proof of Proposition~\ref{prop:codim 1 skeleta}]
To prove (i), assume first that $A$ is a matrix of support $\Sigma$ contained 
in the codimension 1 skeleton of the secondary fan of $\Gamma_\Sigma$.
The regular subdivision $\mc S(A)$ it induces on $\Gamma_\Sigma$ is then not a triangulation,
so it must contain a face $F$ which is not a simplex. The set of vertices of $F$ 
corresponds to a subset $\tau \subseteq \Sigma$. It follows that there exists
$(x,y) \in \R^d \times \R^n$ such that 
\begin{equation}\label{eq:minimum}
\min_{(i,j)\in \Sigma} (x_i-y_j+a_{ij}) 
\text{ is attained precisely when $(i,j) \in \tau$}.
\end{equation}
In view of \cite[Lemma 12.5]{PostPAB},
the graph $\tau$ must contain a simple cycle 
$
C = \{(i_1, j_1), (i_2,j_1), (i_2, j_2), (i_3,j_2), 
\dotsc, (i_s,j_s), (i_1,j_s)\}. 
$
Consider the partial matchings $\lambda  = \{(i_1,j_1),\ldots,(i_s,j_s)\}$ and $\lambda' = \{(i_2,j_1),\ldots,(i_1,j_s)\}$.
It follows from \eqref{eq:minimum} that $A_\lambda = A_{\lambda'} \leq A_{\lambda''}$
for any partial matching $\lambda''$ from $I = \{i_1, \dotsc, i_s\}$ to 
$J= \{j_1, \dotsc, j_s\}$, as desired.

Conversely, suppose $A$ is a matrix supported on $\Sigma$, and
$\lambda  = \{(i_1,j_1),\ldots,(i_s,j_s)\}$ and
$\lambda' = \{(i_2,j_1),\ldots,(i_1,j_s)\}$
are distinct partial matchings jointly attaining the minimum
in some tropical minor of~$A$.
Let $x\in\R^d$ be a point defined as follows.
Adopt the convention that indices on $i$ and $j$
are taken modulo~$s$.
Choose $x_{i_1}$ arbitrarily, and for all $k=2,\ldots,s$, take 
\begin{equation}\label{eq:codim 1 skeleta def y}
x_{i_k} = x_{i_{k-1}} + a_{i_{k-1},j_{k-1}} - a_{i_k,j_{k-1}}  
\end{equation}
For $i$ not of the form $i_\ell$, take $x_i$ very large.
Let $y \in \R^n$ be defined by $y_j = \min_i (a_{ij} + x_i)$.
The minimum in \eqref{eq:minimum} is then equal to 0;
denote $\tau \subseteq \Sigma$ the subset where it is attained.
For any $k\in [s]$ and $i$ distinct from $i_k$, we claim that
\[x_i \geq x_{i_k} + a_{i_k,j_k} - a_{i,j_k} ,\]
so $(i_k,j_k) \in \tau$ and therefore $(i_{k+1},j_k) \in \tau$ too.
If $i$ is not of the form $i_\ell$ this is clear, so suppose $i=i_\ell$
with indexing modulo~$s$ so that $\ell>k$.  By summing inequalities
of type \eqref{eq:codim 1 skeleta def y} we have
\[x_{i_\ell} = x_{i_k} 
+ (a_{i_k,j_k} - a_{i_{k+1},j_k} ) + \cdots 
+ (a_{i_{\ell-1},j_{\ell-1}} - a_{i_\ell,j_{\ell-1}}),\]
so the claim follows as long as the sum of entries of $a$ on the right hand side
is greater than or equal to $a_{i_k,j_k} - a_{i_\ell,j_k}$.
But this is so because the partial matching $\lambda''$ obtained from
$\lambda'$ by replacing the edges $(i_{k+1},j_k)$, \ldots, $(i_\ell,j_{\ell-1})$
with $(i_{k+1},j_{k+1})$, \ldots, $(i_{\ell-1},j_{\ell-1})$, $(i_\ell,j_k)$
satisfies $A_{\lambda''}\geq A_{\lambda'}$.

The subset $\tau \subseteq \Sigma$ contains then both $\lambda$ and $\lambda'$,
so it contains the cycle $\lambda \cup \lambda'$. Again, by
\cite[Lemma 12.5]{PostPAB}, the face $F$ of $\mc S(A)$ it corresponds to 
is not a simplex. The subdivision $\mc S(A)$ is thus not a triangulation,
so $A$ is in the codimension 1 skeleton of the secondary fan of $\Gamma_{\Sigma}$.

To prove (ii), note that $A$ is in the support of the codimension 1 skeleton of the
normal fan of $\Pi_{d,n}(\Sigma)$ if and only if the matching multifield $\Lambda(A)$
is not a matching field, in which case the result is clear.
\end{proof}

\section{Combinatorics of Stiefel tropical linear spaces}\label{sec:Stiefel TLS combinatorics}

In this section we study the combinatorial structure of Stiefel tropical linear spaces.
We prove that their associated dual matroid subdivisions
correspond to ``regular transversal matroid subdivisions''. We also investigate the
connection between these matroid subdivisions and coherent matching multifields.

\subsection{Transversal matroid polytopes}\label{sec:transversal_polytopes}
We start by briefly studying certain inequality descriptions of matroid polytopes of transversal matroids. 
A general inequality description for a matroid polytope is well known \cite{Edmonds}
(see for example \cite{Welsh}):
$\Gamma_M$ is the set of $x\in\R^n$ satisfying the inequalities
$\sum_{j\in J} x_j\leq r_M(J)$ for each $J\subseteq[n]$, 
and the equality $\sum_{j\in[n]} x_j = r_M([n])$,
where $r_M:2^{[n]}\to\mathbb Z$ is the rank function of the matroid~$M$.
Proposition~\ref{r:transversalpolytope} gives a variant description.

Suppose $G \subseteq [d] \times [n]$ is a bipartite graph on the set of vertices $[d] \amalg [n]$,
and let $M= M(G)$ be its associated rank $d$ transversal matroid on the ground set $[n]$.
Recall that for $I \subseteq [d]$, we use the notation $J_I = J_I(G) = \{ j \in [n] : (i,j) \in G \text{ for some } i \in I \}$.

\begin{proposition}\label{r:transversalpolytope}
 The matroid polytope $\Gamma_M$ of the transversal matroid $M = M(G)$ can be described as the set of $x \in \R^n$ satisfying the inequalities:
\begin{align}
 x_1 + x_2 + \dotsb + x_n \, &= \, d, \label{eq:rank}\\
 0 \, \leq x_j \, &\leq \, 1 &\quad \text{ for all $j \in [n]$},\\
 \sum_{j \in J_I} x_j &\geq \, |I| &\quad \text{ for all $I \subseteq [d]$}. \label{eq:trans}
\end{align}
\end{proposition}
\begin{proof}
Let $Q$ be the polytope described by the inequalities given above. If the matroid $M$ is the empty matroid,
Hall's Marriage Theorem implies that there exists $I \subseteq [d]$ such that $|J_I| < |I|$, so the polytope
$Q$ is empty. Suppose now that the matroid $M$ has at least one basis. All vertices of the matroid
polytope $\Gamma_M$ are in $Q$, so we have $\Gamma_M \subseteq Q$. For the reverse inclusion, assume $x \in Q$.
The definition of $Q$ implies that for any $A \subseteq [n]$ and $I \subseteq [d]$, we have the following inequalities:
\begin{align*}
\sum_{j \in [n] \setminus J_I} x_j &\leq d - |I|,\\
\sum_{j \in A \cap J_I} x_j &\leq |A \cap J_I|,\\
\sum_{j \in [n] \setminus (A \cup J_I)} -x_j &\leq 0.
\end{align*}
Adding all these together we get
\[
 \sum_{j \in A} x_j \leq |A \cap J_I| + d - |I|.
\]
Since $I$ was arbitrary, we conclude that
\begin{equation}\label{eq:ranktransversal}
 \sum_{j \in A} x_j \leq \min_{I \subseteq [d]} \left( |A \cap J_I| + d - |I| \right).
\end{equation}
The rank of the subset $A$ in the transversal matroid $M$ is given precisely by the right hand side of Inequality \eqref{eq:ranktransversal} \cite[Proposition 12.2.6]{Oxley}, so we have that $\sum_{j \in A} x_j \leq r_M(A)$. This shows that $x$ satisfies the inequality description stated above 
of the matroid polytope $\Gamma_M$ in terms of $r_M$, completing the proof.
\end{proof}

\begin{corollary}\label{r:transversalcorollary}
 Consider the polytope $P_G = \sum_{i \in [d]} \Delta_{J_i} \subseteq d \cdot \Delta_{[n]}$. Then the matroid polytope $\Gamma_M$ of the transversal matroid $M = M(G)$ is equal to the intersection of $P_G$ with the hypersimplex $\Delta_{d,n} \subseteq d \cdot \Delta_{[n]}$.
\end{corollary}
\begin{proof}
 Inside the hyperplane $\sum_{j \in [n]} x_j = d$, the hypersimplex $\Delta_{d,n}$ is described by the inequalities $0 \leq x_j \leq 1$ for all $j \in [n]$, and the polytope $P_G$ is described by the inequalities $\sum_{j \in X} x_j \geq \left|\{i \in [d] : J_i \subseteq X \} \right|$ for all $X \subseteq [n]$ \cite[Proposition 6.3]{PostPAB}. The result follows from Proposition \ref{r:transversalpolytope}.
\end{proof}

\subsection{Regular transversal matroid subdivisions}
We now describe the combinatorics of Stiefel tropical linear spaces 
in terms of the combinatorial type of $\mc H(A^{\rm t})$ 
(or equivalently, $\mc H(A)$ --- see Remark \ref{rem:transpose_arr}),
and show that the associated matroid subdivisions are 
regular transversal matroid subdivisions.

Consider a matrix $A = (a_{ij}) \in\Rinf^{d\times n}$, 
and assume its support $\Sigma \subseteq [d] \times [n]$ contains a matching.
As described in Section \ref{sec:hyperplanes}, 
the \emph{rows} of the matrix $A$ give rise to a hyperplane arrangement $\mc H(A^{\rm t})$. 
The combinatorial structure of $\mc H(A^{\rm t})$ 
is dual to the associated mixed subdivision $\mc M(A^{\rm t})$ 
of the Minkowski sum of simplices 
\[P_{\Sigma} = \sum_{i=1}^d \Delta_{J_i(\Sigma)} \subseteq d \cdot \Delta_{[n]},\] 
where $J_i(\Sigma) = \{ j \in [n] : a_{ij} \neq \infty\}$. 
Let $p = \pi(A)$ be the tropical Pl\"ucker vector 
obtained by applying the Stiefel map to $A$. 
The vector $p$ corresponds to a tropical linear space $L(A)$, 
whose combinatorial structure is given by the corresponding 
regular matroid subdivision $\mc D(A)$ of the polytope
$\Gamma_p= \conv \{ e_J : p_J \neq \infty \}$.

\begin{theorem}\label{r:restrictsubdiv} 
The matroid subdivision $\mc D(A)$ of the
underlying matroid polytope $\Gamma_p$ is obtained 
by restricting the mixed subdivision $\mc M(A^{\rm t})$ 
of $P_{\Sigma} \subseteq d \cdot \Delta_{[n]}$ to 
the hypersimplex $\Delta_{d,n} \subseteq d \cdot \Delta_{[n]}$.
\end{theorem}
\begin{remark}
It is not true in general that all the faces of $\mc D(A)$ are obtained by intersecting a face of $\mc M(A^{\rm t})$ with $\Delta_{d,n}$, as this could result in a collection of polytopes that is not closed under taking faces. The subdivision $\mc D(A)$ is obtained by taking all possible intersections between a face of $\mc M(A^{\rm t})$ and \emph{a face of} $\Delta_{d,n}$.
\end{remark}
\begin{proof}[Proof of Theorem \ref{r:restrictsubdiv}]
Suppose $Q$ is a cell in the mixed subdivision $\mc M(A^{\rm t})$ indexed by a subgraph
$G \subseteq \Sigma$. It follows that there exists $(x,y) \in \R^d \times \R^n$ such that
$\min_{(i,j) \in \Sigma} ( x_i - y_j + a_{ij} )$ is achieved precisely when $(i,j) \in G$.
Now, assume that $G$ contains a matching of size $d$. The face in the subdivision $\mc D(A)$
corresponding to the vector $y \in \R^n$ is the matroid polytope whose vertices are the indicator
vectors of the subsets $B \in \binom{[n]}{d}$ for which $p_B - \sum_{j \in B} y_j$ is minimal.
Moreover, for any $B$ we have
\begin{align*}
p_B - \sum_{j \in B} y_j &= \min \Bigg\{ \sum_{(i,j) \in \lambda} a_{i j} : \lambda \text{ is a matching on } B \Bigg\} - \sum_{j \in B} y_j \\
&= \min \Bigg\{ \sum_{(i,j) \in \lambda} a_{i j} - y_{j} : \lambda \text{ is a matching on } B \Bigg\} \\
&= \min \Bigg\{ \sum_{(i,j) \in \lambda} a_{i j} - y_{j} + x_i : \lambda \text{ is a matching on } B \Bigg\} - \sum_{i \in [d]} x_i,
\end{align*}
so $\min_B (p_B - \sum_{j \in B} y_j)$ is achieved precisely when $B$ is a subset such that $G$
contains a matching from $[d]$ to $B$. In other words, the face in $\mc D(A)$ selected by $y \in \R^n$
is the matroid polytope of the rank $d$ transversal matroid $M(G)$ associated to the graph $G$. We have thus shown
that for every face $Q$ in $\mc M(A^{\rm t})$ indexed by a subgraph $G \subseteq \Sigma$ containing a maximal matching, 
the matroid polytope of the transversal matroid $M(G)$ is a face of the subdivision $\mc D(A)$.
Note that this also holds for any face $Q$ indexed by a subgraph $G$ that does not contain a maximal matching, in which case
the matroid polytope $M(G)$ is empty.

Now, in view of Corollary \ref{r:transversalcorollary}, the intersection of the polytope $P_{\Sigma}$
with the hypersimplex $\Delta_{d,n}$ is precisely the matroid polytope $\Gamma_p$, and moreover, the intersection
of any face in the mixed subdivision $\mc M(A^{\rm t})$ with $\Delta_{d,n}$ is a face of the matroid subdivision
$\mc D(A)$. Since $P_{\Sigma}$ is completely covered by the faces in the subdivision $\mc M(A^{\rm t})$, 
it follows that all interior faces of the subdivision $\mc D(A)$ are obtained by intersecting a face
of $\mc M(A^{\rm t})$ with $\Delta_{d,n}$, as desired.
\end{proof}

The first part of the proof of Theorem \ref{r:restrictsubdiv} will be useful
for us later, so we record it in a separate proposition.
\begin{proposition}\label{r:transversal}
Let $y\in\R^n$, and 
let $G \subseteq [d] \times [n]$ be the (transpose) 
tropical covector of $y$ in $\mc H(A^{\rm t})$, that is,
\begin{equation}\label{eq:transpose covector}
G = \tc(y)^{\rm t} = \left\{ (i,j_0) \in [d] \times [n] : a_{ij_0}+y_{j_0} = \min_j(a_{ij}+y_j)\right\}.
\end{equation}
If $G$ contains a matching, then
the matroid $M$ selected by $y$ in the regular subdivision $\mc D(A)$
is the transversal matroid $M(G)$ of $G$.
\end{proposition}

The following corollary shows that the matroid subdivisions
dual to Stiefel tropical linear spaces can be
thought of as {\em regular transversal matroid subdivisions}.

\begin{corollary}\label{r:reg_transversal}
The facets of the matroid subdivision $\mc D(A)$ are the matroid polytopes of 
the transversal matroids $M(\tau)$ associated to maximal covectors $\tau \in \TC(A)$.
\end{corollary}
\begin{proof}
Theorem \ref{r:restrictsubdiv} and Corollary \ref{r:transversalcorollary} imply 
that facets of $\mc D(A)$ correspond to the transversal matroids of (transpose) 
maximal covectors in $\TC(A^{\rm t})$. The result follows from Remark \ref{rem:transpose_arr}.
\end{proof}

\begin{example}
Suppose $\Sigma$ is the pointed support set 
(see Examples \ref{ex:pointed} and \ref{ex:(c) -> (a) pointed}). 
Full dimensional subpolytopes of the root polytope $\Gamma_\Sigma$ 
must contain all the vertices of the form $(e_i,-e_j)$ with $i,j \in [d]$.
This implies that for any matrix $A$ supported on $\Sigma$, all facets
of the mixed subdivision $\mc M(A^{\rm t})$ contain the point
$e_1 + \dotsb + e_d \in d \cdot \Delta_{[n]}$. 
In view of Theorem \ref{r:restrictsubdiv}, 
it follows that all facets of the matroid subdivision $\mc D(A)$
contain the vertex $e_1 + \dotsb + e_d$ of $\Delta_{d,n}$,
as was proved in \cite{HJS, LTLS}. Matroid subdivisions
satisfying this property are called conical matroid subdivisions in \cite{LTLS}.
\end{example}

\subsection{Stiefel tropical linear spaces and matching multifields}

We now show that the combinatorial structures
of Stiefel tropical linear spaces and their
corresponding coherent matching multifields are in bijection
in the support set case, thereby providing a combinatorial object
which labels the faces of the Stiefel image.

\begin{theorem}\label{thm:(b) <-> (c)}
Let $\Sigma$ be a support set.
There is a bijection between 
matroid subdivisions in the image of $\pi_\Sigma$
and coherent matching multifields supported on $\Sigma$, under which
$\mc D(A)$ maps to $\Lambda(A)$
for each matrix $A$ of support $\Sigma$.
\end{theorem}

In particular, if $A$ and~$A'$ are matrices of support $\Sigma$,
then $\mc D(A)=\mc D(A')$ if and only if $\Lambda(A)=\Lambda(A')$.

To prove the theorem, we exhibit the two directions of this bijection separately.
Proposition~\ref{prop:(c) -> (b)} gives a construction to pass from
the multifield to the matroid subdivision, and
Corollary~\ref{cor:(b) -> (c)} is the statement that we may 
pass in the other direction.
In fact, Proposition~\ref{prop:(c) -> (b)} holds in greater generality
than just support sets.  The support set assumption is, however,
essential to Corollary~\ref{cor:(b) -> (c)}.

We make a temporary extension of the definition of the matching multifield:
for any matrix $A\in\Rinf^{d\times n}$, let $\tilde\Lambda(A)$ be the set of
matchings $\lambda$ such that $\sum_{(i,j)\in\lambda} a_{ij}$
is minimized among all matchings on the same set of columns.
In the case that $\supp(A)$ contains a support set, $\tilde\Lambda(A) = \Lambda(A)$.

\begin{proposition}\label{prop:(c) -> (b)}
Let $\Sigma$ be any bipartite graph such that $M(\Sigma)$ is a connected matroid,
and every edge of $\Sigma$ is contained in a matching.
If $A$ is a tropical matrix supported on $\Sigma$, then
the maximal faces of~$\mc D(A)$ are
the transversal matroids of the maximal subgraphs $G$ of~$\Sigma$ 
such that $\tilde\Lambda(A)$ contains all matchings in~$G$.
\end{proposition}

\begin{proof}
Let $G$ be a subset of $\Sigma$ containing a matching, such that $\tilde\Lambda(A)$ contains
all matchings contained in~$G$.
Let $G'$ be the set of edges contained in some matching
contained in $G$, so that $M(G')=M(G)$.
The matroid $M(G')$ is the direct sum of all the matroids $M(C)$
where $C$ ranges over the connected components of $G'$.  
We will first construct vectors $t'\in\R^n$ and $u'\in\R^d$ such that
\begin{equation}\label{eq:refstar}
a_{ij} + u'_i - t'_j
\begin{cases}
= 0 & \mbox{for }(i,j)\in G' \\
\geq 0 & \mbox{in any case}. \\
\end{cases}
\end{equation}

Let $F$ be the spanning forest of $G'$ 
provided by Corollary~\ref{cor:spanning tree without left leaves}.
We now augment $F$ to a certain spanning tree $F'$ of~$\Sigma$.
Since $F$ is acyclic, there exist elements
$u\in\R^d$ and $t\in\R^n$ of tropical tori
such that $a_{ij}+u_i-t_j=0$ for all $(i,j)\in F$.
If $G'$ has multiple components, then there is choice
in the values $u_i$ and $t_j$.  To be precise, given any vector
$w\in\R^{\mc C}$ where $\mc C$ is the set of components of $G'$, 
the new vectors
$u'\in\R^d$ and $t'\in\R^n$ given by $u_i' = u_i+w_{C(i)}$
where $C(i)$ is the component containing $i$, and similarly
$t_j' = t_j+w_{C(j)}$ where $C(j)$ is the component containing $j$,
also satisfy 
$a_{ij}-u'_i+t'_j=0$ for all $(i,j)\in F$.

Construct a $\mc C\times\mc C$ matrix $B$ with off-diagonal entries
\[b_{C_1C_2} = \min\{a_{ij} : i\in C_1,j\in C_2\}\]
when $C_1\neq C_2$, and diagonal entries $b_{CC} = \infty$.
Let $\mu$ be the tropical (left) eigenvalue of~$B$,
that is, the minimum mean weight of a cycle in $B$,
interpreting it as a complete weighted directed graph.
We claim that $\mu\geq 0$.  If not, then there is a 
cycle $D$ in $B$ of negative weight.  For each edge $(C_1,C_2)$ of~$D$,
choose an edge $(i,j)$ of~$\Sigma$ so that $a_{ij}$ attains
the minimum in the definition of $b_{C_1C_2}$;
let $D'$ be this set of edges.  
Between each successive pair of edges in~$D'$
there is a path in~$F$ from the right vertex of the former
to the left vertex of the latter.
The union of $D'$ and all of these paths
gives a cycle $\widetilde D$ 
in $\Sigma$ with total weight $\mu <0$,
since paths in $F$ have weight~0.  Note that all edges
of~$D'$ are given the same orientation in the cycle $\widetilde D$.
Just as in the proof of Corollary~\ref{cor:spanning tree without left leaves},
we can construct a partial matching in $F$ by directing the edges of $F$ 
outside $\widetilde D$ away from~$\widetilde D$
and then choosing an out-edge from each left vertex.
Uniting this partial matching with the sets of alternate edges in $\widetilde D$
gives rise to two matchings
$\lambda,\lambda'$: one of these, $\lambda$, contains $D'$
while $\lambda\setminus D'$ is contained in~$F$;
the other, $\lambda'$, is wholly contained in~$F$.
But then $\sum_{(i,j) \in \lambda} a_{ij} = 
\mu + \sum_{(i,j) \in \lambda'} a_{ij} < \sum_{(i,j) \in \lambda'} a_{ij}$, 
so $\lambda'$ is a matching 
contained in $G$ but it is not in $\tilde\Lambda(A)$, contradicting our hypothesis.  
We thus have $\mu\geq0$.  

Suppose $d<n$, for otherwise our proposition is trivial.
Some component $C_*$ of~$G'$ has more right vertices than left ones,
and Lemma~\ref{lem:spanning tree without left leaves} implies
that the transversal matroid $M(C_*)$ is connected, hence coloop-free.  
Define a matrix $B'$ with entries given by $b'_{C_*C_*} = 0$
and $b'_{C_1C_2} = b_{C_1C_2}$ otherwise.
The tropical eigenvalue of $B'$ is zero,
since the cycle on the single vertex $C_*$ is the one of minimum mean weight.  Also, 
$B'$ has a tropical eigenvector $v$ such that,
for every component $C_1$, there is a path
$C_1, C_2,\ldots, C_\ell=C_*$ so that 
$b'_{C_kC_{k+1}} + v_{C_k} = v_{C_{k+1}}$ for all $k$,
and in any case for any two components $C_1,C_2$ we have $b'_{C_1C_2} + v_{C_1} \geq v_{C_2}$.
Fix such a $v$ and a family of such paths, so that for each component $C=C_1$
other than $C_*$ we have fixed a choice of component $s(C)=C_2$.  
Now construct a vector $w\in\R^{\mc C}$ as follows: let $w_{C_*} = 0$, and 
for every component $C\neq C_*$, 
choose $i\in C, j\in s(C)$ so that $a_{ij}$ attains the minimum 
in the definition of $b_{C,s(C)}$, and impose the condition
$w_C - w_{s(C)} = v_C-v_{s(C)}-u_i+t_j$.
These conditions together uniquely determine the vector~$w$, and for the edges $(i,j)$ involved we have
\[a_{ij}+u'_i-t'_j = b_{C,s(C)}+u'_i-t'_j 
= v_{s(C)}-v_C+u_i+w_C-t_j-w_{s(C)} = 0.\]  
Let $F'$ be the 
spanning tree of $\Sigma$ which is the union of $F$
and the set of these individual edges $(i,j)$.
Then by construction we have $a_{ij}+u_i-t_j=0$ when $(i,j)\in F'$.

We now confirm inequalities \eqref{eq:refstar} for $t'$ and~$u'$.
We have just verified equality for edges of~$F'$,
so let $e$ be an edge of $\Sigma\setminus F'$. 
There is a single cycle $Z$ contained in $F'\cup\{e\}$. 
By Corollary~\ref{cor:spanning tree without left leaves},
which we used to construct $F\subseteq F'$,
we see that the sets of alternate edges of $Z$ can be augmented to
two matchings $\lambda$ and $\lambda'$
on the same right vertex set $J$,
such that $\lambda$ contains $e$, and all edges of 
either matching other than~$e$ are drawn from~$F'$.  
Therefore, writing $e=(i_0,j_0)$, we have
\begin{equation}\label{eq:5.9b}
a_{i_0j_0} + u_{i_0}' - t_{j_0}' =
\sum_{(i,j)\in\lambda} a_{ij}+u_i'-t_j' \geq \sum_{(i,j)\in{\lambda'}} a_{ij}+u_i'-t_j' = 0
\end{equation}
since the quantity $a_{ij}+u_i'-t_j'$ is zero on $F'$ 
and $\lambda' \in \tilde\Lambda(A)$ because $\lambda' \subseteq G$.
This is the inequality in~\eqref{eq:refstar}.
The equality for edges $e\in G'$ follows because, in this case,
$\lambda$ is also a matching in $\tilde\Lambda(A)$ and then by definition
equality is attained in~\eqref{eq:5.9b}.

It follows that the bases of $\face_{t'}\mc D(A)$ are all sets $J$
on which there is a matching $\lambda$ with 
$a_{ij}+u_i'-t_j'=0$ for all $(i,j)\in\lambda$.
If there was an edge $(i,j)\not\in G$ such that
$a_{ij} + u'_{i} - t'_{j} = 0$,  
then the graph $G''=G\cup\{(i,j)\}$ would share with $G$ the property
that every matching it contains is in $\tilde\Lambda(A)$:
this is because $\sum_{(i,j)\in\lambda} a_{ij} + u'_{i} - t'_{j}$
is zero for every matching $\lambda$ contained in $G''$
(whether or not $(i,j)$ features), 
but is nonnegative for any matching $\lambda$ contained in $\Sigma$.
So if $G$ is maximal with the said property, 
then every edge $(i,j)$ with $a_{ij} + u'_{i} - t'_{j}=0$
must be contained in $G$.
It follows that $\face_{t'}\mc D(A)$ equals $M(G)$.

Finally, to show that if $G$ is maximal then $M(G)$ is connected
and therefore a maximal face of $\mc D(A)$, 
let $H$ be a set of edges of $\Sigma$ composed of,
for each pair $C_1\neq C_2$ of components of $G'$ such that $b'_{C_1C_2} + v_{C_1} = v_{C_2}$,
one edge from a left vertex $i\in C_1$ to a right vertex $j\in C_2$
with $a_{ij} + u'_i - t'_j = 0$.
Observe that every edge of~$H$ is in~$G$.
We claim that each edge of $H$ is contained in some matching on $\Sigma$.
Indeed, each individual edge $e$ of $H$ lies at the beginning of a path $p$ 
contained in $F'\cup H$ whose end lies in $C_*$ and is the only vertex
of~$p$ in~$C_*$, and such that 
all the edges of $p$ in $H$ are traversed from left to right.
Then, since $M(C_*)$ is coloop-free, there is a matching on $C_*$
avoiding the endpoint of $p$.
By the argument in the proof of Corollary~\ref{cor:spanning tree without left leaves},
the edges which $p$ traverses right to left can be extended to
a matching on the union of the other components of $G'$.
The desired matching containing $e$ is then
the union of the edges which $p$ traverses left to right
and these matchings on $C_*$ and the other components of $G'$.  It follows from
Lemma~\ref{lem:spanning tree without left leaves} that $M(F'\cup H)$
is connected.  Since $F'\cup H$ contains every vertex of $\Sigma$
and $G\supseteq F'\cup H$, every basis of $M(F'\cup H)$ is a basis of $M(G)$,
and thus $M(G)$ is connected too.

All that remains is to show that every maximal face of $\mc D(A)$
in fact is produced by our construction.  
By Corollary~\ref{r:reg_transversal},
every maximal face $M$ in $\mc D(A)$
is the transversal matroid of a graph of the form $\tc(y)^{\rm t}$,
and $\tilde\Lambda(A)$ contains all matchings contained in this graph.
If $G$ is a maximal graph with this property,
such that $\tc(y)^{\rm t}\subseteq G$, then the bases of $M$
are a subset of those of $M(G)$.  But we have shown that $M(G)$ 
is some maximal face of $\mc D(A)$,
so $M(G)$ must equal $M$.
\end{proof}

Suppose $\Sigma$ is a support set and $A$ is a tropical matrix supported on $\Sigma$.
To prove the other direction of Theorem~\ref{thm:(b) <-> (c)},
we describe the procedure for recovering the graphs $\tc(y)^{\rm t}$, as
in Proposition~\ref{r:transversal}, that arise for the matroids in $\mc D(A)$.
The collection of all sets of edges forming a matching in
any one of these graphs is the matching multifield $\Lambda(A)$ of $A$.

For $i \in [d]$, consider the functional $f_i=\sum_{j\in J_i(\Sigma)} y_j$, and
let $F_i$ be the face of the hypersimplex $\Delta_{d,n}$ minimizing $f_i$.
Observe that $F_i$ is the translate of the simplex $\Delta_{J_i(\Sigma)}$
by $\sum_{j\not\in J_i(\Sigma)} e_j$.

\begin{lemma}\label{lem:(b) -> (c) 1}
The face of $P_\Sigma$ minimizing $f_i$ contains $F_i$.
Thus, the face of $\mc D(A)$ minimizing $f_i$ equals $F_i$.
\end{lemma}

\begin{proof}
We first establish that 
the minimum value attained by the functional $f_i=\sum_{j\in J_i(\Sigma)} y_j$
on $P_\Sigma$ is~1.  
For this, the minimum attained by a functional is additive under
Minkowski sum.  
The minimum attained by $f$ on $\Delta_{J_i(\Sigma)}$ is~1. 
On the other summands $\Delta_{J_{i'}(\Sigma)}$
of~$P_\Sigma$ the mininum is 0, because Proposition~\ref{r:SZ3.1}(c)
with $I = \{i,i'\}$ and the definition of support set
imply that $|J_i\cup J_{i'}| \geq n-d+2 > n-d+1 = |J_{i'}|$,
thus that $J_{i'}$ does not contain $J_i$. 

The first claim of the lemma for any $j\in J_i(\Sigma)$,
the point $x=\sum_{j'\in [n]\setminus J_i(\Sigma)\cup\{j\}} e_{j'}$
is a vertex of~$P_\sigma$ (at which clearly $f_i$ assumes the value~1).
This follows because
the set $\Sigma$ supports a matching $\lambda$ on $[n]\setminus J_i(\Sigma)\cup\{j\}$,
so that $x$ is in~$P_\Sigma$ by the Minkowski sum description, 
being the sum of the points
$e_{j'}\in\Delta_{J_{i'}(\Sigma)}$
for each $(i',j')\in\lambda$.

The second claim then follows from Theorem~\ref{r:restrictsubdiv}.
Since the underlying space of $\mc D(A)$ is the intersection of $P_\Sigma$
with $\Delta_{d,n}$, and $\mc D(A)$ contains $\face_{f_i}\Delta_{d,n}=F_i$ as a face,
this must also be $\face_{f_i}\mc D(A)$.
\end{proof}

Let $\mc D(A)^*$ be the subdivision normal to the regular subdivision $\mc D(A)$:
that is, the open cells of $\mc D(A)^*$ are the sets 
$\{u : \face_u\mc D(A) = F\}$ for faces $F$ of $\mc D(A)$.
By Lemma~\ref{lem:(b) -> (c) 1}, $F_i$ is a face of $\mc D(A)$.
Using the restriction of the height function on $\mc D(A)$ 
to~$F_i$ realises $F_i$ as a trivial
regular subdivision with one maximal face.
Its own normal subdivision,
which we will call $F_i^*$,
is a translation of the normal fan $\mc N(F_i)$, 
lying in the same ambient space that contains $\mc D(A)^*$.

If $M$ is a connected matroid whose polytope $\Gamma_M$ 
is a cell of $\mc D(A)$, define for each $i\in[d]$ 
a subset $J_i$ of~$J_i(\Sigma)$ as follows.
Let $\Gamma_M^*$ be the vertex of $\mc D(A)^*$ dual to $\Gamma_M$.
This vertex is contained in the relative interior of some 
face of $F_i^*$; take $J_i$ to be the subset of $J_i(\Sigma)$
indexing this face. 
Now define the bipartite graph $\widetilde G(M)$ on vertex set $[d]\amalg[n]$
so that the set of neighbours of the left vertex $i$
is $J_i$ for each~$i$.

\begin{proposition}\label{prop:(b) -> (c)}
Let $\Sigma$ be a support set, and let $M$ be a connected matroid contained in
a regular matroid subdivision $\mc D$ in the image of $\pi_\Sigma$.
Then the graph $\widetilde G(M)$ is well defined,
depending only on $M$ and $\mc D$.
Also, if $u$ is a vector such that $M$ equals the face $\face_u(\mc D)$ of $\mc D$ selected by $u$,
then $\widetilde G(M)$ is the graph $\tc(u)^{\rm t}$.
\end{proposition}

\begin{corollary}\label{cor:(b) -> (c)}
Let $\Sigma$ be a support set, and $A$ a tropical matrix of support $\Sigma$.
Then the matching multifield $\Lambda(A)$ is
determined by the matroid subdivision $\mc D(A)$.
\end{corollary}

\begin{proof}
$\Lambda(A)$ is the union of all matchings contained
in the graphs $\tc(u)^{\rm t}$.
In fact, it is sufficient to take the union of
the $\tc(u)^{\rm t}$ where $u$ is a functional
selecting a full-dimensional cell of $\mc D(A)$,
since if $u$ selects a cell of positive codimension
in $\mc D(A)$ and $v$ selects a full-dimensional cell of which $u$
is a face, then $\tc(u)^{\rm t}\subseteq \tc(v)^{\rm t}$.
Proposition~\ref{prop:(b) -> (c)} ensures that the graphs $\tc(u)^{\rm t}$
for full-dimensional cells of $\mc D(A)$, 
that is for connected matroids, are determined by $\mc D(A)$.
\end{proof}

\begin{proof}[Proof of Proposition~\ref{prop:(b) -> (c)}]
Let us fix some tropical matrix $A$ 
so that $\mc D(A)$ is the matroid subdivision chosen.
The dual subdivision to $\mc M(A^{\rm t})$ is $\mc H(A^{\rm t})$.
The $i$\/th hyperplane $H_i(A^{\rm t})$ is the
dual to the simplex $\Delta_{J_i(\Sigma)}$, interpreted 
as a trivial regular subdivision with its vertex $j$ given the height $a_{ij}$.
That is, $H_i(A^{\rm t})$ is the translate of the normal
fan $\mc N(\Delta_{J_i(\Sigma)})$
of $\Delta_{J_i(\Sigma)}$, taken to be based at the origin,
by the row vector $a^{<\infty}_i\in\R^{J_i(\Sigma)}$ 
formed by the non-infinite components of the $i$\/th row of~$A$.
(By this translation, we mean the translation by the vector in $\R^n$
obtained by interpolating zeroes for indices lacking in $a_i^{<\infty}$.)

Now let $b_i\in\R^{J_i(\Sigma)}$ be the row vector such that $b_j$
is the $([d], [n]\setminus J_i(\Sigma)\cup\{j\})$ tropical minor of~$A$ 
for each $j\in J_i(\Sigma)$.
Then every component of the vector $b_i - a_i^{<\infty}$ 
equals the $([d]\setminus\{i\}, [n]\setminus J_i(\Sigma))$ tropical minor of~$A$,
by expanding the determinants along row~$i$. 
Since this difference is contained in the lineality space of the fan $\mc N(\Delta_{J_i(\Sigma)})$,
the translations of $\mc N(\Delta_{J_i(\Sigma)})$ 
by $a_i^{<\infty}$ and by~$b_i$ are equal.
The former translation is $H_i(A^{\rm t})$,
whereas the latter is $F_i^*$, since 
the heights in the regular subdivision $\mc D(A)$
and its regular sub-subdivision $F_i$ are the tropical maximal
minors of~$A$.
So $H_i(A^{\rm t}) = F_i^*$.

Theorem~\ref{r:restrictsubdiv} indicates that $\mc D(A)$ is the
restriction of $\mc M(A^{\rm t})$ to the hypersimplex,
in the sense described there.  If $Q$ is a full-dimensional
cell of $\mc M(A^{\rm t})$ whose restriction $\Gamma_M$ in $\mc D(A)$
is also full-dimensional, then the corresponding cells in 
the dual subdivisions $\mc D(A)^*$ and $\mc H(A^t)$ are both points, and these points
are equal because a functional realizing equal (minimum)
values on all vertices of $Q$ does the same on $\Gamma_M$.
Name the point which constitutes these cells $u$.

By definition the $i$\/th row of the covector
$\tc(u)^{\rm t}$ is the set indexing the face of 
$F_i^* = H_i(A^{\rm t})$ containing~$u$.
Therefore $\widetilde G(M)$ is the covector $\tc(u)^{\rm t}$.

Finally, we assert that $\tc(u)^{\rm t}$ can be determined from 
the cell complex $\mc D=\mc D(A)$ alone, without the data
of the lifting heights in the regular subdivision.
It will follow that $\widetilde G(M)$ is independent of the choice of~$A$.

Consider a ray in $\mc D(A)^*$ emanating from $u$ in direction $f_i$.
Let $E^*$ be the face of $\mc D(A)^*$ in which points of this ray
sufficiently far from $u$ lie.
Then $E^*$ is dual to a face of $\mc D(A)$ on which $f_i$
takes only its minimum value.  By Lemma~\ref{lem:(b) -> (c) 1}
this is a face $E$ of $F_i$.  
In particular, $E$ equals $\face_{u+zf_i}(\mc D(A))$ for $z\gg0$ real,
which is $\face_u(\face_{f_i}(\mc D(A))) = \face_u(F_i)$.
Labelling the faces of the simplex $F_i$ with subsets of $J_i(\Sigma)$
in the standard way, the label of~$E$
is the $i$\/th row of $\tc(u)^{\rm t}$.  

What we wish to show is that it is still possible to carry out this procedure
to determine the $i$\/th row of $\tc(u)^{\rm t}$ using only the
data in the combinatorial type of the subdivision $\mc D(A)^*$.
We will in fact show that the following walking procedure fills the bill:
walk along the faces of $\mc D(A)^*$,
beginning from the vertex of interest, and passing at each step to one of
the faces you could reach next by moving in direction $f_i$,
until reaching a face dual to a face $\Delta_J$ of $F_i$.
Then $J$ is the $i$\/th row of $\tc(u)^{\rm t}$.

The walking procedure we have just sketched does in fact work
when applied to the complex $\mc H(A^{\rm t})$: 
indeed, there is a subcomplex of $\mc H(A^{\rm t})$
with the same support as $F_i^*$,
and it is clear
that a walk as described will never cross from the interior of
one face of $F_i^*$ into a different one,
since $f_i$ is contained in the lineality space of $F_i^*$.

Now, by Theorem~\ref{r:restrictsubdiv}, $\mc D(A)$
is obtained from $\mc M(A^{\rm t})$ by restricting it to the intersection
of all of the halfspaces
$\{x_j\leq 1\}$ for each $j\in [n]$.
Given a polyhedral complex $S$ and its restriction $S'$ to a halfspace $\{\langle x,f\rangle\leq a\}$
in the sense of Theorem~\ref{r:restrictsubdiv}, the relationship 
between the dual complexes $S^*$ and $(S')^*$ is that every cell of $(S')^*$
is either a cell of $S^*$ or is the Minkowski sum of a cell of $S^*$
with a ray in direction $-f$; these correspond to the primal cells
which respectively don't meet and meet the halfspace.

From this description, it follows that the dual complex
of the restriction of $\mc M(A^{\rm t})$ to the intersection of the
halfspaces $\{x_j\leq 1\}$ for $j\not\in J_i(\Sigma)$ 
still contains a subcomplex with total space $F_i^*$,
so that the walking procedure still works there.

Now consider restricting to the remaining halfspaces,
those $\{x_j\leq 1\}$ with $j\in J_i(\Sigma)$.
The walk we construct starts at a vertex $u$ of $\mc D(A)^*$,
which manifestly contains no ray $\R_+(-y_j)$.
We will also show, by induction, that none of the faces we 
may encounter on our walk contain such a ray,
so that the walk proceeds exactly as it did on $\mc H(A^{\rm t})$.
Indeed, if $F$ contains no ray in any of these directions, 
then certainly none of its faces do.  
As for the other kind of step, 
if a face $G$ of $F$ contains no ray in direction $-y_j$
but $F$ does, then a vector in direction $y_j$ based at $G$
points (strictly) out of $F$.
On the other hand, if our walk takes us from $G$ to $F$,
then a vector in direction $f_i$ based at $G$ points (strictly) into $F$.
Therefore, there must be some $\lambda\in(0,1)$ such that
\[v=\lambda y_j + (1-\lambda)f_i = y_j + \sum_{j'\in J_i(\Sigma)\setminus\{j\}} y_{j'} \]
is contained in the linear span of a facet $F'$ of $F$ containing $G$.  
However, the normal vectors to any facet appearing in
the complexes at hand is of form $y_{j_1}-y_{j_2}$, for $j_1,j_2\in[n]$:
this is true of $\mc H(A^{\rm t})$ and remains true as
we perform restrictions.  If $v$ is contained in 
the lineality space of such a facet, i.e.\
$v_{j_1} = v_{j_2}$ are equal, then by the last display
either $j_1$ and $j_2$ are both in $[n]\setminus J_i(\Sigma)$ 
or both in $J_i(\Sigma)\setminus\{j\}$.  In either case $y_j$
and $f_i$ are contained in the same lineality space, i.e.\ they point
along the boundary of $F$ from~$G$, which is a contradiction.
\end{proof}

\section{Tropical linear maps and their images}\label{sec:tropical linear maps}

In this section we relate Stiefel tropical linear spaces and tropical hyperplane arrangements
by investigating the parametrization of a tropical linear space $L(A)$ coming from matrix-vector multiplication.
Throughout this section we will assume $A \in \Rinf^{d \times n}$ is a tropical matrix 
where no column has all its entries equal to $\infty$.

\begin{definition}
Let ${\odot A}:\Rinf^d\to\Rinf^n$ denote tropical right multiplication by~$A$,
where $\Rinf^d$ and $\Rinf^n$ are regarded as spaces of row vectors.
\end{definition}

If $F$ is a face of the hyperplane arrangement complex $\mc H(A)$,
we will write $\tc(F)$ for the tropical covector $\tc(x)$ of any point $x$ in its relative interior.
Denote by $\mc B(A)$ the subcomplex of all cells
$F$ in $\mc H(A)$ such that
for each~$i\in[d]$ there exists some $j \in [n]$ such that $(i,j) \in \tc(F)$.
Geometrically, if $\mc H(A)$ has no hyperplanes at infinity,
the cells in~$\mc B(A)$ are those containing no ray in any of the
standard basis directions.  If $\mc H$ contains hyperplanes
at infinity with supports $\{i_1\}, \ldots, \{i_k\}$, then
rays in the directions $e_{i_1}, \ldots, e_{i_k}$ should
be omitted from consideration.
For instance, if $A$ has pointed support then $\mc B(A) = \mc H(A)$.

The following proposition generalizes Theorem 23 in \cite{DS} to arbitrary matrices in $\Rinf^{d\times n}$.
Its proof is based on the ideas discussed in Remark \ref{rem:transpose_arr}.

\begin{proposition}\label{r:odot A}
The map $\odot A$ is a piecewise linear map from $\TT^{d-1}$ to~$\TT^{n-1}$
whose domains of linearity are the facets of the cell complex $\mc H(A)$.
The restriction of~$\odot A$ to the underlying space $|\mc B(A)|$
is a bijection $(\odot A)|_{\mc B(A)}:|\mc B(A)|\to\im(\odot A)$.
\end{proposition}
\begin{proof}
The domains of linearity of the coordinate
$(x\odot A)_j$ are the facets $F_i(H_j)$ of the fan
associated to the hyperplane $H_j$.  It follows that
the domains of linearity of~$x\odot A$ are the facets of~$\mc H(A)$.

If $x\in\TT^{d-1}$ is such that $J_i(\tc(x))=\emptyset$
then the minimum in~$(x\odot A)_j$ is not attained by $x_i+a_{ij}$ for any~$j$.
For such~$x$
we may change $x_i$ to $\min_{j\in[n]}(x\odot A)_j-a_{ij}$
without changing $x\odot A$.  Thus $(\odot A)|_{\mc B(A)}$
surjects onto $\im(\odot A)$.

Given $y\in\im(\odot A)$, we have $x_i\geq y_j-a_{ij}$ for any
point $x\in(\odot A)^{-1}(y)$ and all $i,j$.
If also $x\in|\mc B(A)|$ then for each~$i\in[d]$ there is some
$j\in[n]$ such that $x_i=y_j-a_{ij}$.  In this situation
$x_i$ is uniquely determined to be~$\max_{j\in[n]}y_j-a_{ij}$
for each~$i$.  That is, there is a unique point in
each fiber $(\odot A)^{-1}(y)$ lying in~$|\mc B(A)|$, so $(\odot A)|_{\mc B(A)}$ is injective.
\end{proof}

Theorem~2.1 of~\cite{STY} relates the image of the tropicalization
of an algebraic map~$f$ to the tropicalization of its image.
To wit, the latter can be constructed from the former by 
taking Minkowski sums of certain faces with orthants.    
Theorem~\ref{r:folding cells} is essentially a generalization of this result in the linear case
to nontrivial valuations, describing $L(A)$ in terms of the image of ${\odot A}$.
The addition in Equation \eqref{eq:L'(A)} denotes Minkowski sum,
and the notation $\R_{\geq 0}S$, for $S$ a subset of a real vector space,
denotes the cone of nonnegative linear combinations of elements of~$S$.

\begin{theorem}\label{r:folding cells}
The tropical linear space $L(A)$ (in the tropical torus $\TT^{n-1}$)
equals the union
\begin{equation}\label{eq:L'(A)}
 L(A) = \bigcup_{\substack{F \in \mc B(A) \\ J \in \mc J_F}} 
 \bigl( F \odot A \enskip + \enskip \R_{\geq 0}\{e_j:j\in J\} \bigr),
\end{equation}
where $\mc J_F$ denotes the collection of subsets $J \subseteq [n]$ such that 
\begin{equation}\label{cd:J}
\mbox{for every nonempty $J'\subseteq J$, \enskip $\left|I_{J'}(\tc(F))\right|\geq |J'|+1$.}
\end{equation}
\end{theorem} 

The combinatorial condition on the sets $J$ that we have stated as condition~\eqref{cd:J}
has a more algebraic guise in~\cite{STY}.  Namely, the statement is that
the ideal $(\mathrm{in}_u(\cl f_j) : j\in J)$
contains no monomial, where $u$ is a relative interior point of~$F$,
and $\cl f_j$ are generic linear forms 
$\sum (\cl f_j)_i \, \cl x_i$ in variables $\cl x_i$
tropicalizing to the tropical linear forms determining the map $\odot A$, 
so that the matrix $\cl A = [(\cl f_j)_i]_{i,j}$ tropicalizes to~$A$.
The entries $(\cl a_j)_i$ which survive in the $u$-initial form are 
exactly those with $(i,j)\in\tc(F)$. 
A linearly generated ideal is prime, so if it contains a monomial,
it must contain a variable, one of the $\cl x_i$, which will be
a scalar combination of the $\cl f_j$.  
Lemma~\ref{lem:kernel meets torus} asserts that condition~\eqref{cd:J}
is exactly the condition under which this is generically avoided.

\begin{lemma}\label{lem:kernel meets torus}
Let $\bb K$ be a field, and $G\subseteq[d]\times J$ for some set $J$.
The (left) kernel of a generic matrix $\cl A\in\bb K^{d\times J}$ of support~$G$
fails to be contained in any coordinate hyperplane of $\bb K^d$ if and only if,
for every nonempty $J'\subseteq J$ we have $|I_{J'}(G)| \geq |J'|+1$ or $|I_{J'}(G)| = 0$.
\end{lemma}

\begin{proof}
Note that the $(I,J')$ minor of $\cl A$ will be generically nonzero so long
as its permutation expansion contains any surviving term once the
entries outside $G$ are set to zero, that is, if $G$ contains a matching
from $I$ to $J'$.

Suppose some $J'$ has $0 \neq |I_{J'}(G)| < |J'|+1$.  Choose a minimal such $J'$;
it follows that $|I_{J'}(G)| = |J'|$, since otherwise an element of $J'$ 
could be deleted while maintaining the inequality.  
Hall's marriage theorem (applied transposely) shows that there is a partial matching
from $I_{J'}(G)$ to $J'$ contained in $G$.  
Therefore, the $(I_{J'}(G), J')$ submatrix of a generic $\cl A$ is invertible; 
denote its inverse by $\cl B$.  The columns of the product of
the $([d], J')$ submatrix of $\cl A$ with~$\cl B$
are the standard basis vectors $e_i$ for each $i\in I_{J'}(G)$,
so these $e_i$ are in the column space of $\cl A$, and thus 
the kernel of $\cl A$ is contained in $\{\cl x:\cl x_i=0\}$
for all such $i$.

Conversely, if $\ker\cl A$ is contained in $\{\cl x : \cl x_i=0\}$ for some~$i$,
i.e.\ if $e_i$ is in the column space of $\cl A$, write it as 
a linear combination of the columns, 
\[e_i = \sum_j\cl c_j\cl a_j,\]
and let $J^* = \{j : \cl c_j\neq 0\}$.  
Since the coefficients $\cl c_j$ are not all zero,
they cannot satisfy more than $|J^*|-1$ independent linear relations.
But if $|I_{J'}(G)| \geq |J'|+1$ for all nonempty $J'\subseteq J^*$, then 
the displayed equation imposes $|J^*|$ independent linear relations on them,
because the dragon marriage theorem (applied transposely)
shows the existence of a matching from a subset of~$I_{J^*}(G)\setminus\{i\}$ 
to $J^*$, and thus the corresponding $|J^*|\times|J^*|$ minor of $\cl A$ is nonzero.
\end{proof}

We may also see \eqref{cd:J} as a tropical transverse intersection condition. 
Indeed, our condition is the (transpose) dragon marriage condition on the induced subgraph
with vertices $I_J(\tc(F)) \amalg J$,
which by Proposition~\ref{r:SZ2.4} is equivalent to the existence 
of an acyclic subgraph of $\tc (F)$ with degree 2 at every~$j\in J$. 
This in turn can be translated to the geometric condition that
one can choose a facet of $H_j$ containing $F$ for each $j\in J$,
in such a way that their affine spans intersect transversely.

\begin{example}\label{ex:folding}
 Consider the matrix
 \[
  A = \begin{pmatrix}
        0 & 0 & \infty & \infty \\
        \infty & 0 & 0 & \infty \\
        \infty & \infty & 0 & 0
      \end{pmatrix}.
 \]
The corresponding tropical Pl\"ucker vector $\pi(A) \in \Gr(3,4)$ is the all zeroes vector,
so the tropical linear space $L(A)$ is the standard tropical plane in 3-space with vertex at the origin.
The behavior of the tropical linear map $\odot A : \TT^2 \to \TT^3$ is depicted in Figure \ref{f:linear_map}.
The linear space $L(A)$ can be obtained from $\im( \odot A)$ by adding different orthants, as described
by Theorem \ref{r:folding cells}.
\begin{figure}[ht]
\centering
\includegraphics[scale=0.35]{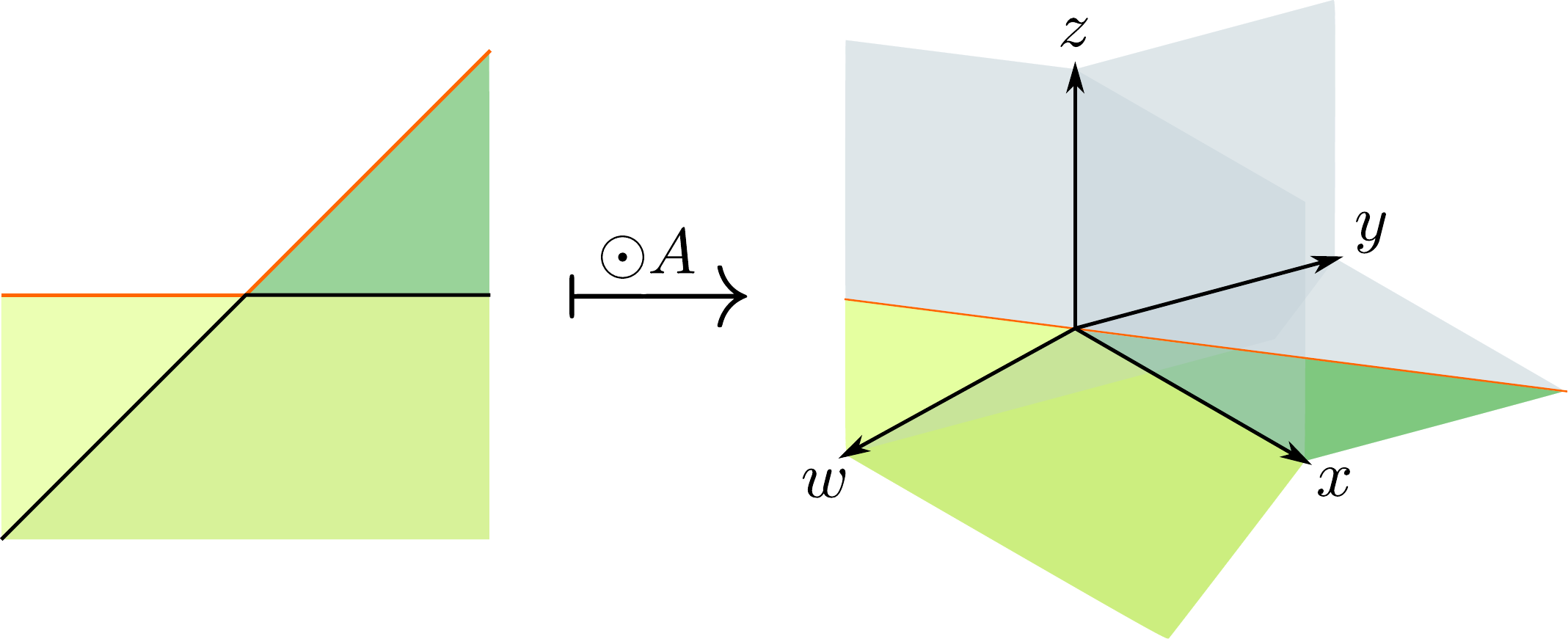}
\caption{At the left, the hyperplane complex $\mc H(A)$ for the matrix $A$ of Example \ref{ex:folding}.
The colored regions correspond to the subcomplex $\mc B(A)$.
This subcomplex gets mapped homeomorphically by $\odot A$
into the linear space $L(A)$, as is shown at the right.
The different colors are meant to illustrate the correspondence among different faces
under this map. The gray colored cones are the extra orthants added according to Theorem \ref{r:folding cells}.}
\label{f:linear_map}
\end{figure}
\end{example}

We provide a proof of Theorem~\ref{r:folding cells} using
elementary tropical geometry.  This sidesteps the need to
generalize the geometric tropicalization machinery of~\cite{STY} 
to nontrivially valued fields.

\begin{proof}[Proof of Theorem~\ref{r:folding cells}]
Fix $\bbK$ to be a field of generalized power series in $t$ 
over a residue field~$\bbk$ with value group $\R$, 
so that $\bbk\hookrightarrow\bbK$. 
Let $\cl A\in\bbK^{d\times n}$ be a generic classical matrix with 
$\val(\cl A)=A$, so that $L(A)$ is the tropicalization of $\cl L(\cl A)$. 
The fundamental theorem of tropical geometry tells us that 
$L(A)=\val(\cl L(\cl A))$.

Given $y\in L(A) \cap \R^n$, choose a classical $\cl y\in\cl L(\cl A)$
with $\val(\cl y)=y$ so that $\cl y=\cl x \cdot \cl A$
for some $\cl x\in(\bbK^*)^d$. Let $x = \val(\cl x)$, and define 
$y' = x \odot A \in L(A)$.
We have $y_j\geq y'_j$ for each~$j \in [n]$.
Let $J$ be the set of indices $j$ such that the inequality $y_j>y'_j$ is strict. 
For each $j\in J$, there is a cancellation among the leading terms of the sum
$\cl y_j = \sum_{i} \cl x_i\cl a_{ij}$,
that is, the leading coefficients satisfy 
\[
\sum_{i\in I_j(\tc(x))}\lc(\cl x_i)\lc(\cl a_{ij})=0.
\]
Viewed as a set of linear equations in the unknowns $\lc(\cl x_i)\in\bbk$,
these have a solution in the torus $(\bbk^\ast)^d$, since leading
coefficients of nonzero power series must be nonzero.
Let $\cl B = (\cl b_{ij}) \in \bbk^{d \times J}$ be the matrix defined by
\begin{equation}\label{eq:submatrix}
\cl b_{ij} = 
\begin{cases}
\lc(\cl a_{ij}) & \text{if } i\in I_j(\tc(x)), \\
0 & \text{otherwise.}
\end{cases}
\end{equation}
The support of $\cl B$ is the subgraph of $\tc(x)$
consisting of all edges between the vertices $[d] \amalg J$,
and $\lc(\cl x)\cdot \cl B$ is the zero vector.
Since no component of $\lc(\cl x)$ may be zero, 
Lemma~\ref{lem:kernel meets torus} shows that $J$ satisfies condition~\eqref{cd:J}.
Moreover, following the proof of Proposition \ref{r:odot A}, we can find a point
$x' \in |\mc B(A)|$ such that $x' \odot A = x \odot A$ and $\tc(x') \supseteq \tc(x)$. 
It follows that $y \in (x' \odot A) + \R_{> 0}\{e_j:j\in J\}$ 
where $\tc(x')$ and $J$ satisfy the desired conditions, so
$y$ is a point in the right hand side of \eqref{eq:L'(A)}.

Conversely, assume $y \in (x \odot A) + \R_{> 0}\{e_j:j\in J\}$,
where $x$ is in the relative interior of some face $F$ in $\mc B(A)$, and 
$F$ and $J$ satisfy condition \eqref{cd:J}.
We wish to construct $\cl x\in\bbK^d$ with valuation~$x$
such that $\val(\cl x\cdot \cl A)=y$.  
Write $\cl x_i = (\lc(\cl x_i) + \cl x'_i)t^{x_i}$.
Again, consider the matrix $\cl B = (\cl b_{ij}) \in \bbk^{d \times J}$ defined by
\eqref{eq:submatrix}.
In view of Lemma~\ref{lem:kernel meets torus},
it is possible to choose the leading coefficients 
$\lc(\cl x_i) \in \bbk^*$ of~$\cl x$ such that
$\lc(\cl x)\cdot \cl B$ is zero.
By our genericity assumption on $\cl A$, we may suppose that 
$\sum_{i\in I_j(\tc(x))}\lc(\cl x_i)\lc(\cl a_{ij})$ is nonzero
for any $j \in [n] \setminus J$, ensuring that $\val((\cl x \cdot \cl A)_j)=y_j$
for any $j \in [n] \setminus J$.

Now, for each $j \in J$ choose a generic $\cl y_j\in\bbK$ with $\val(\cl y_j)=y_j$.  
Solving for the higher-order parts $\cl x'_i$ of the $\cl x_i$
which will arrange that $\sum_i \cl x_i \cl a_{ij}= \cl y_j$ for all $j \in J$
is a question of solving a linear system over $\bbK$
of $|J| < d$ equations in~$d$ unknowns. After dividing each of these equations
by the corresponding term $t^{(x \odot A)_j}$, we get a system whose coefficients 
have nonnegative valuation and whose constant terms have strictly positive valuation.
The coefficients that have zero valuation are precisely those corresponding to
the indices $(i,j) \in \tc(x)$.
Hall's marriage theorem implies that there is a partial matching from $I_J(\tc(x))$ to $J$ 
contained in $\tc(x)$. We can perform Gaussian elimination to get pivots in the entries corresponding to
this matching, and our genericity assumption ensures that in this process we only have to invert 
elements of~$\bbK$ of zero valuation. 
This shows that the system has a solution $(\cl x'_i)$ where all coordinates have positive valuation,
as desired.
\end{proof}

It is tempting to imagine that the terms in the 
union in Equation~\eqref{eq:L'(A)} would be the closed cells
of a cell complex, given that Proposition~\ref{r:odot A}
shows that the $F\odot A$ form a cell complex,
and the orthants added over each individual $F\odot A$ 
together with their faces clearly do.
Example~\ref{ex:Felipe March 5}
is a cautionary one, showing that there can be non-facial
intersections between these putative cells
(though it is interesting 
that the covectors in the example do not contain any matchings).

\begin{example}\label{ex:Felipe March 5}
Take $(d,n) = (11,12)$, and let $\Sigma = \{(i,i),(i,i+1) : i=1,\ldots,11\}$
be a path graph.  Let $A\in\Rinf^{d\times n}$ be the tropical matrix of support $\Sigma$ 
with all finite entries equal to zero (there is only one orbit of
matrices of support $\Sigma$ under the two tropical tori, anyhow).

Let $C_1$ be the term in the union in~\eqref{eq:L'(A)}
arising when $F$ is the face of $\mc B(A)$ with covector 
$\Sigma\setminus\{(4,5),(6,6),(8,9),(10,10)\}$
and $J$ equals $\{3,4,11\}$, and let $C_2$ be the term arising when
$F$ has covector $\Sigma\setminus\{(2,2),(4,5),(6,6),(10,10)\}$
and $J$ is $\{7,8,11\}$.  Inequality descriptions of these two polyhedra are
\begin{align*}
C_1 &= \{y_3,y_4\ge y_1=y_2\ge y_5=y_6;\ y_7=y_8\ge y_6,y_9;\ y_{11}\ge y_{12}\ge y_9=y_{10}\}, \\
C_2 &= \{y_1=y_2;\ y_3=y_4\ge y_2,y_5;\ y_7,y_8,y_{12}\ge y_9=y_{10}\ge y_5=y_6;\ y_{11}\ge y_{12}\}.
\end{align*}
Each of the two has codimension~4 in $\R^{12}$.
However, the intersection of $C_1$ and $C_2$ is a polyhedron of codimension~5
contained in the relative interiors of both:
\[
C_1\cap C_2 = \{y_3=y_4\ge y_1=y_2\ge y_5=y_6;\ y_7=y_8\ge y_9=y_{10}\ge y_6;\ y_{11}\ge y_{12}\ge y_9\}. \qedhere
\]
\end{example}

We now use these results to investigate the part of the cell complex $\mc B(A)$ that gets
mapped by $\odot A$ to the bounded part of $L(A)$.
We first give a lemma that presents a local condition describing the faces of a tropical linear space (not necessarily Stiefel) that are bounded in $\TP^{n-1}$.  

\begin{lemma}\label{lem:bounded}
Let $L$ be a tropical linear space whose underlying matroid is the uniform matroid $U_{d,n}$, and let $x \in L$.
Then $x$ is in the bounded part of $L$ if and only if for all $j \in [n]$ and all $\varepsilon >0$ small enough
we have $x - \varepsilon \cdot e_j \notin L$.
\end{lemma}
\begin{proof}
The bounded part of $L$ corresponds to the faces of $L$ that are dual to interior cells in the associated matroid subdivision $\mc D$.
The subdivision $\mc D$ is a subdivision of the hypersimplex $\Delta_{d,n}$, so interior faces are precisely the
faces not contained in any of the hyperplanes $x_i = 0$ or $x_i = 1$, that is, faces corresponding to loop-free and coloop-free matroids.

Now, let $M$ be the matroid in $\mc D$ selected by the vector $x$, i.e. the matroid
whose bases are the subsets $B \in \binom{[n]}{d}$ for which $p_B - \sum_{i \in B} x_i$ is minimal,
where $p$ is the tropical Pl\"ucker vector corresponding to $L$. The matroid $M_j$ corresponding to the
vector $x - \varepsilon \cdot e_j$ consists of all the bases of $M$ whose intersection with $\{j\}$
is as small as possible, so $j$ is a coloop in $M$ if and only if $j$ is not a loop in $M_j$.
Since $M$ is loop free, the only possible loop $M_j$ could have is $j$. It follows that $M$ is
coloop free if and only if all the matroids $M_j$ have loops.
Because a matroid in $\mc D$ is dual to a cell of~$L$ if and only if 
it has no loops, this completes the proof.
\end{proof}

We now describe the bounded part of $L(A)$ in terms of the complex $\mc H(A)$.
Denote by $\mc K(A)$ the subcomplex of $\mc H(A)$ consisting of all faces $F$ whose tropical covector satisfies 
\begin{equation}\label{eq:coloopless}
\text{ for every nonempty } I \subseteq [d], \quad |J_I(\tc(F))| \geq |I| + 1.
\end{equation}
Note that condition \eqref{eq:coloopless} is simply the dragon marriage condition on the graph $\tc(F)$.

\begin{theorem}\label{r:bounded}
Assume the support $\Sigma$ of $A$ contains a support set.
Then the map $\odot A$ restricts to a piecewise linear homeomorphism between $\mc K(A)$ 
and the subcomplex of $L(A)$ consisting of all its bounded faces, 
which is an isomorphism of polyhedral complexes.
\end{theorem}
\begin{proof}
The subcomplex $\mc K(A)$ is contained in $\mc B(A)$, so Proposition \ref{r:odot A}
ensures that $\odot A$ is injective on $|\mc K(A)|$. Condition \eqref{eq:coloopless} defining $\mc K(A)$
implies that if $x \in |\mc K(A)|$ then its tropical covector $\tc(x)$ contains a matching,
so Proposition \ref{r:transversal} tells us that $x \odot A$ is in a face $F$ of $L(A)$ dual to the
matroid polytope of the transversal matroid of the graph $\tc(x)$. 

We next argue that \eqref{eq:coloopless}
implies that this matroid, call it $M(x)$, has neither loops nor coloops, so that $F$ is indeed an
interior cell of the subdivision $\mc D(A)$ and thus $x \odot A$ is in the bounded part of $L(A)$.
Given any element $j\in[n]$, the collection of sets 
\[\{J_i(\tc(F))\setminus\{j\} : i\in [d]\}\]
supports a matching by \eqref{eq:coloopless} and Hall's theorem, which provides a basis of~$M(x)$
not containing $j$, so that $j$ is not a coloop.  Similarly, $I_j(\tc(F))$ is nonempty, 
and if $i_0$ is any of its elements, then the collection
\[\{J_i(\tc(F))\setminus\{j\} : i\in [d]\setminus\{i_0\}\}\]
supports a matching which extends to a matching in $\tc(F)$ by insertion of $(i_0,j)$,
implying that $j$ is not a loop of~$M(x)$.

We will now prove that $(\odot A)|_{\mc K(A)}$ surjects onto the bounded part of $L(A)$.
Theorem \ref{r:folding cells} and Lemma \ref{lem:bounded} imply that the bounded part of $L(A)$
is contained in the image of $\odot A$. We will show that if $x \in |\mc B(A)|$ but $x \notin |\mc K(A)|$
then $x \odot A$ is not in the bounded part of $L(A)$. 
Let $x \in |\mc B(A)| \setminus |\mc K(A)|$, so there exists a nonempty 
$I \subseteq [d]$ such that
$|J_I(\tc(x))| \leq |I|$. Let $G$ be the induced subgraph of $\tc(x)$ 
on the vertices $I \amalg J_I(\tc(x))$.
Consider the transversal matroid $N$ on the set $J_I(\tc(x))$
associated to the graph $G$, whose independent sets correspond to subsets of $J_I(\tc(x))$ that can
be matched in $G$ with a subset of $I$. 
The matroid $N$ must have coloops: either $G$ has a complete matching
in which case all elements are coloops, or by Hall's theorem it has an inclusion-minimal set $I'$
such that $|I'|>|J_{I'}(\tc(x))|$.  In the latter case, by minimality, 
$|I'| = 1+|J_{I'}(\tc(x))|$ and there exists a matching $\lambda$ 
to $J_{I'}(\tc(x))$ from a subset of $I'$ of the same size, which is nonempty because $x\in|\mc B(A)|$.  
Then given any partial matching on $N$, removing any edges with left vertices in $I'$ 
and then inserting edges of $\lambda$ as necessary gives a new matching whose set of right
vertices contains $J_{I'}(\tc(x))$: i.e.\ $J_{I'}(\tc(x))$ consists of coloops.

Denote now by $M$ be the matroid in $\mc D(A)$ selected by the vector $y = x \odot A$.
Let $j_0 \in J_I(\tc(x))$ be a coloop of $N$.
We claim that $j_0$ is also a coloop of $M$, and thus $M$ is not an interior cell of $\mc D(A)$.
Assume by contradiction that $j_0$ is not a coloop of $M$, 
and let $B$ be a basis of $M$ not containing $j_0$.
Just as in the proof of Theorem \ref{r:restrictsubdiv}, it follows that there exists a matching
$\lambda$ in $\Sigma$ from $[d]$ to $B$ of minimal weight
$\sum_{(i,j) \in \lambda} a_{ij} +x_i - y_j$ among all matchings in $\Sigma$.
Consider a maximal partial matching $\lambda'$ in the graph $G$
(of size possibly smaller than $|I|$ and $|J_I(\tc(x))|$).
The symmetric difference of the matchings $\lambda$ and 
$\lambda'$ is a union of cycles and paths whose edges
alternate between $\lambda$ and $\lambda'$. 
Since $j_0$ is an endpoint of an edge in $\lambda'$ but not an
endpoint of any edge in $\lambda$, one of these alternating paths 
$\ell$ starts with an edge $(j_0, i_1) \in \lambda'$
and ends with an edge of $\lambda$. 
The edges of $\lambda'$ are all in $\tc(x)$, and all the edges of $\ell$ that are part of $\lambda$ must
also be in $\tc(x)$, otherwise the matching obtained by taking $\lambda$ and swapping all the edges
in $\lambda \cap \ell$ for the edges in 
$\lambda' \cap \ell$ would contradict the minimality of the weight of
$\lambda$. It follows that all the edges of $\ell$ 
are edges of the graph $G$. But then, taking the matching
$\lambda'$ and swapping all the edges in $\lambda' \cap \ell$ 
for the edges in $\lambda \cap \ell$ gives rise
to a maximal matching of $G$ that does not include the vertex $j_0$, 
contradicting that $j_0$ was a coloop of $N$.

Finally, in order to prove that $\odot A$ is an isomorphism of polyhedral complexes,
it is enough to show that any two different tropical covectors labelling faces of $\mc K(A)$
have different associated transversal matroids. Assume this is not the case. Since $\odot A$
is a homeomorphism, we can find two distinct tropical covectors 
$\sigma, \tau \in \TC(A)$ satisfying condition 
\eqref{eq:coloopless} and having the same transversal matroid, such that the face of $\mc K(A)$
labelled by $\tau$ is a face of the face labelled by $\sigma$, i.e., $\sigma \subsetneq \tau$.
It follows that there exist partitions $[d] = I_1 \amalg \dotsb \amalg I_r$, 
$[n] = J_1 \amalg \dotsb \amalg J_r$ so that $\tau$ is a subset 
of $\bigcup_{k \geq \ell} I_k \times J_\ell$ and $\sigma$ is the 
intersection of $\tau$ with $\bigcup_k I_k \times J_k$ (this is called the 
{\em surrounding property} in \cite{TropOMs}). Let $(i,j)$ be an edge in $\tau \setminus \sigma$.
We thus have $i \in I_s$ and $j \in J_t$ for some $s>t$. By \cite[Theorem 3]{BondyWelsh}, since 
$\sigma$ and $\tau$ define the same transversal matroid $M = M(\sigma)$, 
the element $j$ must be a coloop of the deletion of $J_i(\sigma)$ from $M$.
This deletion $M'$ is the transversal matroid of the 
induced subgraph $\sigma'$ of $\sigma$ on the vertices
$([d] \setminus \{i\}) \amalg ([n] \setminus J_i(\sigma))$.
Our assumptions imply that the induced subgraph on the vertices
$I_t \amalg J_t$ is the same for both graphs $\sigma$ and $\sigma'$,
so it follows that $j$ is also a coloop of $M$ and thus $\sigma$
does not label a face in $\mc K(A)$.
\end{proof}

\bibliographystyle{amsalpha}
\bibliography{bibliography}
\label{sec:biblio}

\end{document}